\documentclass[12pt,oneside,reqno]{amsart}

\usepackage{amsmath,amssymb,amsthm,textcomp}
\usepackage{amsfonts,graphicx}
\usepackage[mathscr]{eucal}
\usepackage{color}
\usepackage{csquotes}
\usepackage[backend=bibtex,%
firstinits=true,%
doi=false,%
isbn=true,%
url=false,%
maxnames=99]{biblatex}%

\AtEveryBibitem{\clearfield{issn}}
\AtEveryCitekey{\clearfield{issn}}
\numberwithin{equation}{section}
\DeclareNameAlias{sortname}{last-first}
\theoremstyle{definition}
\usepackage{mathtools}
\numberwithin{equation}{section}

\newcommand{\ncom}{\newcommand}

\ncom{\beq}{\begin{equation}}
\ncom{\eeq}{\end{equation}}
\ncom{\bea}{\begin{eqnarray*}}
\ncom{\eea}{\end{eqnarray*}}
\ncom{\beqa}{\begin{eqnarray}}
\ncom{\eeqa}{\end{eqnarray}}
\ncom{\nno}{\nonumber}
\ncom{\non}{\nonumber}
\ncom{\ds}{\displaystyle}
\ncom{\half}{\frac{1}{2}}
\ncom{\mbx}{\makebox{.25cm}}
\ncom{\hs}{\mbox{\hspace{.25cm}}}
\ncom{\rar}{\rightarrow}
\ncom{\Rar}{\Rightarrow}
\ncom{\noin}{\noindent}
\ncom{\bc}{\begin{center}}
\ncom{\ec}{\end{center}}
\ncom{\sz}{\scriptsize}
\ncom{\rf}{\ref}
\ncom{\s}{\sqrt{2}}
\ncom{\sgm}{\sigma}
\ncom{\Sgm}{\Sigma}
\ncom{\psgm}{\sigma^{\prime}}
\ncom{\dt}{\delta}
\ncom{\Dt}{\Delta}
\ncom{\lmd}{\lambda}
\ncom{\Lmd}{\Lambda}
\ncom{\Th}{\Theta}
\ncom{\e}{\eta}
\ncom{\eps}{\epsilon}
\ncom{\pcc}{\stackrel{P}{>}}
\ncom{\lp}{\stackrel{L_{p}}{>}}
\ncom{\dist}{{\rm\,dist}}
\ncom{\sspan}{{\rm\,span}}
\ncom{\re}{{\rm Re\,}}
\ncom{\im}{{\rm Im\,}}
\ncom{\sgn}{{\rm sgn\,}}
\ncom{\ba}{\begin{array}}
\ncom{\ea}{\end{array}}
\ncom{\hone}{\mbox{\hspace{1em}}}
\ncom{\htwo}{\mbox{\hspace{2em}}}
\ncom{\hthree}{\mbox{\hspace{3em}}}
\ncom{\hfour}{\mbox{\hspace{4em}}}
\ncom{\vone}{\vskip 2ex}
\ncom{\vtwo}{\vskip 4ex}
\ncom{\vonee}{\vskip 1.5ex}
\ncom{\vthree}{\vskip 6ex}
\ncom{\vfour}{\vspace*{8ex}}
\ncom{\norm}{\|\;\;\|}
\ncom{\integ}[4]{\int_{#1}^{#2}\,{#3}\,d{#4}}
\ncom{\vspan}[1]{{{\rm\,span}\{ #1 \}}}
\ncom{\dm}[1]{ {\displaystyle{#1} } }
\ncom{\ri}[1]{{#1} \index{#1}}

\newtheorem{theorem}{\bf Theorem}[section]
\newtheorem{remark}{\bf Remark}[section]
\newtheorem{proposition}{Proposition}[section]

\newtheorem{corollary}{Corollary}[section]

\newtheoremstyle
    {remarkstyle}
    {}
    {11pt}
    {}
    {}
    {\bfseries}
    {:}
    {     }
    {\thmname{#1} \thmnumber{#2} }

\theoremstyle{remarkstyle}

\def\eps{\varepsilon}

\begin{document}
\title{Generalized Counting Process: its Non-Homogeneous and Time-Changed Versions}
\author[Kuldeep Kumar Kataria]{Kuldeep Kumar Kataria}
\address{Kuldeep Kumar Kataria, Department of Mathematics, Indian Institute of Technology Bhilai, Raipur 492015, India.}
\email{kuldeepk@iitbhilai.ac.in}
\author[Mostafizar Khandakar]{Mostafizar Khandakar}
\address{Mostafizar Khandakar, Department of Mathematics,
	Indian Institute of Technology Bombay, Powai, Mumbai 400076, India.}
\email{mostafizar@math.iitb.ac.in}
\author[Palaniappan Vellaisamy]{Palaniappan Vellaisamy}
\address{Palaniappan Vellaisamy, Department of Mathematics,
		Indian Institute of Technology Bombay, Powai, Mumbai 400076, India.}
\email{pv@math.iitb.ac.in}
\date{October 8, 2022}
\begin{abstract}
 We introduce a non-homogeneous version of the generalized counting process (GCP), namely, the non-homogeneous generalized counting process (NGCP). We time-change the NGCP by an independent inverse stable subordinator to obtain its fractional version, and call it as the non-homogeneous generalized fractional counting process (NGFCP). A generalization of the NGFCP is obtained by time-changing the NGCP with an independent inverse subordinator. We derive the system of governing differential-integral equations for the marginal distributions of the increments of NGCP, NGFCP and its generalization. Then, we consider the GCP time-changed by a multistable subordinator and obtain its L\'evy measure, associated Bern\v{s}tein function and distribution of the first passage times. The GCP and its fractional version, that is, the generalized fractional counting process when time-changed by a L\'evy subordinator are known as  the time-changed generalized counting process-I (TCGCP-I) and the time-changed generalized fractional counting process-I (TCGFCP-I), respectively.  We obtain the distribution of first passage times and related governing equations for the TCGCP-I. An application of the TCGCP-I to ruin theory is discussed. We obtain the conditional distribution of the $k$th order
statistic from a sample whose size is modelled by a  particular case of TCGFCP-I, namely, the time fractional negative binomial process. Later, we consider a fractional version of the TCGCP-I and obtain the system of differential equations that governs its state probabilities. Its mean, variance, covariance, {\it etc.} are obtained and using which its long-range dependence property is established. Some results for its  two particular cases are obtained.
\end{abstract}
\maketitle
\section{Introduction}
Di Crescenzo {\it et al.} (2016) introduced and studied a L\'evy process $\{M(t)\}_{t\ge0}$, namely, the generalized counting process (GCP). Its transition probabilities are given by 
\begin{equation}\label{transition}
	\mathrm{Pr}\{M(t+h)=n|M(t)=m\}=\begin{cases*}
		1-\Lambda h+o(h), \ n=m,\\
		\lambda_{j}h+o(h), \ n=m+j, \ j=1,2,\dots,k,\\
		o(h), \ n>m+k,
	\end{cases*}
\end{equation}
where $\lambda_{j}$'s are positive rates,  $\Lambda=\lambda_{1}+\lambda_{2}+\dots+\lambda_{k}$ and $o(h)\to 0$ as $h\to 0$. From \eqref{transition}, it is evident that the GCP performs $k$ kinds of jumps of amplitude $1,2,\dots,k$ with rates $\lambda_{1}, \lambda_{2},\dots,\lambda_{k}$, respectively. Its fractional version, namely, the generalized fractional counting process (GFCP) is obtained by time-changing the GCP by an independent inverse stable subordinator $\{Y_{\beta}(t)\}_{t\ge0}$, $0<\beta<1$, that is, $\{M\left(Y_{\beta}(t)\right)\}_{t\ge0}$. We denote the GFCP by $\{M^{\beta}(t)\}_{t\ge0}$. Its state probabilities $p^{\beta}(n,t)=\mathrm{Pr}\{M^{\beta}(t)=n\}$ satisfy the following system of fractional differential equations (see Di Crescenzo {\it et al.} (2016)):
\begin{equation}\label{cre}
\frac{\mathrm{d}^{\beta}}{\mathrm{d}t^{\beta}}p^{\beta}(n,t)=-\Lambda p^{\beta}(n,t)+	\sum_{j=1}^{n\wedge k}\lambda_{j}p^{\beta}(n-j,t), \ n\ge0,
\end{equation}
with the initial conditions
\begin{equation*}
	p^{\beta}(n,0)=\begin{cases}
		1, \ n=0,\\
		0, \ n\ge1.
	\end{cases}
\end{equation*}
Here, $\dfrac{\mathrm{d}^{\beta}}{\mathrm{d}t^{\beta}}$ is the Caputo fractional derivative defined in \eqref{caputo} and $n\wedge k=\min\{n,k\}$.

 For $\beta=1$, the GFCP reduces to the GCP. For $k=1$, the GFCP and the GCP reduces to the time fractional Poisson process (TFPP) (see Mainardi {\it et al.} (2004), Beghin and Orsingher (2009)) and the Poisson process, respectively.  Recently, several authors introduced and studied many counting processes such as the Poisson process of order $k$ (PPoK) (see Kostadinova and Minkova (2013)), P\'olya-Aeppli process of order $k$ (PAPoK) (see Chukova and Minkova (2015)), Bell-Touchard process (see Freud and Rodriguez (2022)), convoluted Poisson process (see Kataria and Khandakar (2021)), Poisson-logarthmic process (see Sendova and Minkova (2018)), {\it etc.} These processes along with their fractional variants are particular cases of the GFCP (see Kataria and Khandakar (2022b); Khandakar and Kataria (2022)). So, the  GFCP is a fractional generalization of the Poisson process. For other fractional generalizations of the Poisson process such as the space fractional Poisson process, space-time fractional Poisson process (STFPP), we refer the reader to Orsingher and Polito (2012). The fractional versions of the Poisson process has applications in several fields. For example, in the field of optics to describe the light propagation through non-homogeneous media, in queueing theory (see Cahoy {\it et al.} (2015)), in finance, in modeling the catastrophic events such as tsunami and earthquake, in fractional quantum mechanics (see Laskin (2009)), {\it etc.}

Let $\stackrel{d}{=}$ denotes equality in distribution. Also, let $\{X_{i}\}_{i\ge1}$ be a sequence of independent and identically distributed (iid) random varibles with the following probability mass function (pmf): 
\begin{equation}\label{xj}
	\mathrm{Pr}\{X_{1}=j\}=\lambda_{j}/\Lambda, \ j=1,2,\dots,k.
\end{equation}
 Di Crescenzo {\it et al.} (2016) showed that the GCP is equal in distribution to the following compound Poisson process:
\begin{equation}\label{compound}
	M(t)\stackrel{d}{=}\sum_{i=1}^{N(t)}X_{i},
\end{equation}
where the Poisson process $\{N(t)\}_{t\ge0}$ with intensity $\Lambda$ is independent of $\{X_{i}\}_{i\ge1}$.

Leonenko {\it et al.} (2017), (2019) introduced and studied a non-homogeneous version of the TFPP which is defined as a non-homogeneous Poisson process time-changed by an independent inverse stable subordinator.  Buchak and Sakhno (2019) obtained the governing equations for marginal distributions of a non-homogeneous Poisson process time-changed by inverse subordinator. Recently, Kadankova {\it et al.} (2021) studied a non-homogeneous version of the PPoK and the PAPoK along with their fractional extensions. 

 A multistable subordinator  $\{H(t)\}_{t\ge0}$ with time-dependent stability index $\beta(t)\in (0,1)$ is an inhomogeneous subordinator.
The Laplace transform of increments of an inhomogeneous subordinator
$\{R(t)\}_{t\ge0}$ is given by 
\begin{equation*}
	\mathbb{E}\left(e^{-u(R(t)-R(s))}\right)=e^{-\int_{s}^{t}f(u,\tau)\mathrm{d}\tau}, \ 0\le s\le t,
\end{equation*}
 which holds true under some suitable conditions (see Orsingher {\it et al.} (2016)). Here, $u\to f(u,t)$ is a Bern\v{s}tein function for each $t\ge0$ such that
\begin{equation}\label{bernstein}
	f(u,t)=\int_{0}^{\infty}(1-e^{-ux})\nu_{t}(\mathrm{d}x).
\end{equation}
A multistable subordinator is characterized by the following  L\'evy measure (see Beghin and Ricciuti (2019)):
\begin{equation}\label{nut}
	\nu_{t}(\mathrm{d}x)=\frac{\beta(t)x^{-\beta(t)-1}}{\Gamma(1-\beta(t))}\mathrm{d}x,\ x>0,
\end{equation}
where $t\to\beta(t)$ is a regular function. Its associated  Bern\v{s}tein function is $f(u,t)=u^{\beta(t)}$. The multistable subordinators extend the stable subordinators by considering the time-dependent stability index. These are proved to be useful in modeling phenomena in the field of finance and natural sciences. For more details on multistable subordinator, we refer the reader to Orsingher {\it et al.} (2016). Beghin and Ricciuti (2019) studied a time-changed Poisson process where the multistable subordinator is used as a time-changed component. 

Kataria and Khandakar (2022a) studied a time-changed version of the GFCP, namely, the time-changed  generalized fractional counting process-I (TCGFCP-I) $\{M_{f}^{\beta}(t)\}_{t\ge0}$. It is defined as follows :
\begin{equation*}
	M_{f}^{\beta}(t)\coloneqq M^{\beta}(D_{f}(t)),
\end{equation*}
where the GFCP is independent of the L\'evy subordinator $\{D_{f}(t)\}_{t\ge0}$. For $\beta=1$, it reduces to the time-changed  generalized counting process-I (TCGCP-I) $\{M_{f}(t)\}_{t\ge0}$ whose distribution of jumps is given by 
\begin{equation}\label{bgd43}
	\mathrm{Pr}\{M_{f}(h)=n\}=\begin{cases*}
		1-hf(\Lambda)+o(h), \ n=0,\\
		\displaystyle-h\sum_{ \Omega(k,n)}f^{(z_{k})}(\Lambda)\prod_{j=1}^{k}\frac{(-\lambda_{j})^{x_{j}}}{x_{j}!}+o(h), \ n\ge1,
	\end{cases*}
\end{equation}
where $\Omega(k,n)$ is given in \eqref{p(n,t)} and $z_{k}=x_{1}+x_{2}+\cdots+x_{k}$.

In Section \ref{section2}, we give some preliminary results that will be used later. In Section \ref{section3}, it is shown that if one GCP event has occurred till time $t\ge0$, then
the distribution of time of first GCP event is uniform. An alternate representation of the one-dimensional distribution of GFCP in terms of the derivatives of Mittag-Leffler function is obtained. We introduce a non-homogeneous version of the GCP, namely, the non-homogeneous generalized counting process (NGCP), and obtain its probability generating function (pgf) and the system of differential equations that governs its state probabilities. We time-change the NGCP by an independent inverse stable subordinator to obtain its fractional version, and call it as the non-homogeneous generalized fractional counting process (NGFCP). The systems of governing differential-integral equations for the marginal distributions of increments of these processes are derived. Also, we discuss the arrival times of NGFCP. A  generalization of the NGFCP is introduced by time-changing the NGCP with an independent inverse subordinator, for which we obtain the governing system of equations of marginal distributions that involves convolution-type derivative introduced and studied by Kochubei (2011) and Toaldo (2015).

In Section \ref{section4}, we consider the GCP time-changed by an independent multistable subordinator, that is, $\{M(H(t))\}_{t\ge0}$. It is observed that the process $\{M(H(t))\}_{t\ge0}$ is a non-homogeneous subordinator. We obtain the L\'evy measure, the Bern\v{s}tein function and the distribution of its first passage times. For a particular case of the process $\{M(H(t))\}_{t\ge0}$, we obtain its one-dimensional distribution and pgf. 

In Section \ref{section5}, we obtain some additional results for the TCGCP-I. The system of differential equations that governs its state probabilities, the distribution of its first passage times and related governing equations are obtained. For a particular case of TCGCP-I, we obtain the system of differential equations that governs its state probabilities. As an application, we consider a risk process in which the TCGCP-I is used to model the number of claims received by an insurance company. Some results for the joint distribution of the
time to ruin and the deficit at the time of ruin are derived.  We have shown that the one-dimensional distributions of TCGFCP-I are not infinitely divisible. We discuss the conditional distribution of the $k$th order
statistic from a sample whose size is modelled by a  particular case of TCGFCP-I, namely, the time fractional negative binomial process (see Vellaisamy and Maheshwari (2018b)).

In Section \ref{section6}, we consider the following time-changed process:
\begin{equation*}
	\mathscr{M}_{f}^{\alpha}(t)\coloneqq M_{f}(Y_{\alpha}(t))=M\left(D_{f}(Y_{\alpha}(t))\right), \ t\ge0,
\end{equation*}
where the L\'evy subordinator $\{D_{f}(t)\}_{t\ge0}$, the
inverse stable subordinator $\{Y_{\alpha}(t)\}_{t\ge0}$, $0<\alpha<1$ and the GCP $\{M(t)\}_{t\ge0}$ are independent of each other. We obtain its pgf, mean, variance, covariance and the system of differential equations that governs its state probabilities.  Its long-range dependence (LRD) property is established. We proved some results for two particular cases of this time-changed process. It is known that the process that exhibits LRD property has applications in several areas such as finance, econometrics, hydrology, internet data traffic modeling, {\it etc.}
\section{Preliminaries}\label{section2}
In this section, we collect some preliminary results for Mittag-Leffler function and some known processes. These will be used later.

\subsection{Mittag-Leffler function}
The two-parameter Mittag-Leffler function is defined as (see Kilbas {\it et al.} (2006), p. 42)
\begin{equation*}
	E_{\alpha,\beta}(x)\coloneqq\sum_{j=0}^{\infty} \frac{x^{j}}{\Gamma(j\alpha+\beta)}, \ x\in\mathbb{R},
\end{equation*}
where $\alpha>0$ and $\beta>0$. 

It reduces to Mittag-Leffler function for $\beta=1$. The following identity holds (see Beghin and Orsingher (2009), Eq. (5.1))
\begin{equation}\label{mittag}
	E_{-\beta,1-\beta}(1/x)=-xE_{\beta,1}(x),\ 0<\beta<1,\ x\neq 0.
\end{equation}
\subsection{Generalized counting process} 

The state probabilities $p(n,t)=\mathrm{Pr}\{M(t)=n\}$ of GCP are given by (see Di Crescenzo {\it et al.} (2016))
\begin{equation}\label{p(n,t)}
	p(n,t)=\sum_{\Omega(k,n)}\prod_{j=1}^{k}\frac{(\lambda_{j}t)^{x_{j}}}{x_{j}!}e^{-\Lambda t},\ \ n\ge 0,
\end{equation}
where $\Omega(k,n)\coloneqq\{(x_{1},x_{2},\dots,x_{k}):\sum_{j=1}^{k}jx_{j}=n,\ x_{j}\in\mathbb{N}_0\}$, $\mathbb{N}_0$ is the set of non-negative integers.

Its pgf $G(u,t)=\mathbb{E}\left(u^{M(t)}\right)$ is given by (see Kataria and Khandakar (2022b))
\begin{equation}\label{pgfmt}
	G(u,t)=\exp\Big(-\sum_{j=1}^{k}\lambda_{j}(1-u^{j})t\Big),  \ |u|\le 1.
\end{equation}
The following limiting result holds true for GCP (see Kataria and Khandakar (2022b), Eq. (17)):
\begin{equation}\label{limit}
	\lim\limits_{t\to \infty}\frac{M(t)}{t}=\sum_{j=1}^{k}j\lambda_{j},\ \text{in probability}.
\end{equation}
\subsection{Subordinator and its inverse}
Let 
\begin{equation}\label{fx}
	f(s)=c_{1}+c_{2}s+\int_{0}^{\infty}\left(1-e^{-sx}\right)\overline{\nu}(\mathrm{d}x),  \ s>0,\ c_{1}\ge0,\ c_{2}\geq0,
\end{equation}
 be a Bern\v stein function and $\overline{\nu}(\cdot)$ be a non-negative  L\'evy measure on $(0,\infty)$ such that
  $\displaystyle\int_0^\infty (x\wedge1)\,\overline{\nu}(\mathrm{d}x)<\infty$.  A subordinator $\{D_f(t)\}_{t\ge0}$ is a one-dimensional L\'evy process  which is characterized by the following Laplace transform (see Applebaum (2009), Section 1.3.2)):
\begin{equation*}
	\mathbb{E}\left(e^{-sD_f(t)}\right)=e^{-tf(s)}.
\end{equation*}

Its first passage time is called the inverse subordinator. It is defined as
\begin{equation}\label{Yf}
	Y_f (t)\coloneqq\inf\{x\ge0: D_f (x)> t\}, \ t\ge0.
\end{equation}

Let $h_{f}(t,x)$ be the density of $\{Y_{f}(t)\}_{t\ge0}$. Its Laplace transform is given by (see  Toaldo (2015)):
\begin{equation}\label{lhft}
	\tilde{h}_{f}(s,x)=\int_{0}^{\infty}e^{-st}h_{f}(t,x)\mathrm{d}t=\frac{f(s)}{s}e^{-xf(s)},
\end{equation}
provided that the following condition holds:

\medskip
\noindent   Condition {\bf I.}
$\overline{\nu}(0,\infty)=\infty$ and the tail $\nu(x)=c_{1}+\overline{\nu}(x,\infty)$ of the L\'evy measure $\overline{\nu}(\cdot)$ is absolutely continuous.

Kochubei (2011) and Toaldo (2015) discussed a new type of differential operator that generalizes the Caputo fractional derivative and Riemann-Liouville fractional
	derivative. Let $u(\cdot)$ be an absolutely continuous function. Its generalized Caputo derivative associated with the Bern\v{s}tein function $f$ is defined as follows (see Toaldo (2015), Definition 2.4):
\begin{equation}\label{fDt}
^f{\mathcal D}_t u(t)\coloneqq c_{2} \frac{d}{dt}u(t)+\int_{0}^{t}\frac{\partial }{\partial t }u(t-x)\nu (x)\mathrm{d}x.
\end{equation}
 It is related to  the generalized Riemann-Liouville
derivative $^f\mathscr{D}_t$ as follows  (see Toaldo (2015), Proposition 2.7):
\begin{equation}\label{relation}
	^f\mathscr{D}_t u(t)= ^f\mathcal{D}_t u(t)+\nu(t) u(0).
\end{equation}

If  $|u(t)|\leq C e^{s_0t}$ where $C$ and $s_0$ are some constants, then the Laplace transform of generalized Caputo derivative is given by (see Toaldo (2015), Lemma 2.5)
\begin{equation}\label{LfDt}
	\mathcal{L}\left(^f\mathcal{D}_t u(t);s\right)=f(s)\tilde{u}(s)-\frac{f(s)}{s}u(0), \  s>s_0.
\end{equation}

For $f(s)=s^{\beta}$, $\beta\in (0,1)$, the derivative \eqref{fDt} reduces to (see Toaldo (2015), Remark 2.6):
\begin{equation*}
^f\mathcal{D}_t u(t)=\frac{\mathrm{d}^{\beta}}{\mathrm{d}t^{\beta}}u(t),
\end{equation*}
where $\dfrac{\mathrm{d}^{\beta}}{\mathrm{d}t^{\beta}}$ is the Caputo fractional
derivative defined as (see Kilbas {\it et al.} (2006))
\begin{equation}\label{caputo}
	\frac{\mathrm{d}^{\beta}}{\mathrm{d}t^{\beta}}u(t):=\left\{
	\begin{array}{ll}
		\dfrac{1}{\Gamma{(1-\beta)}}\displaystyle\int^t_{0} (t-s)^{-\beta}u'(s)\,\mathrm{d}s,\  0<\beta<1,\\\\
		u'(t),\ \beta=1.
	\end{array}
	\right.
\end{equation}
Its Laplace transform is given by (see Kilbas {\it et al.} (2006), Eq. (5.3.3))
\begin{equation}\label{lc}
	\mathcal{L}\Big(\frac{\mathrm{d}^{\beta}}{\mathrm{d}t^{\beta}}u(t);s\Big)=s^{\beta}\tilde{u}(s)-s^{\beta-1}u(0), \ s>0.
\end{equation}

\subsection{Gamma subordinator}
The probability density function $g(x,t)$ of a gamma subordinator $\{Z(t)\}_{t\ge0}$ is given by
\begin{equation*}
	g(x,t)=\frac{a^{bt}}{\Gamma(bt)}x^{bt-1}e^{-ax}, \ x>0,
\end{equation*}
where $a>0$ and $b>0$. Its associated Bern\v stein function is (see Applebaum (2009), p. 55)
\begin{equation}\label{gamma}
	f_{1}(s)=b\log(1+s/a), \ s>0.
\end{equation}
For $x\ge0$, $t\geq 0$, the density function of gamma subordinator satisfies (see Beghin (2014), Lemma 2.1)
\begin{equation}
	\left\{
	\begin{array}{l}
		\frac{\partial }{\partial x}g(x,t)=-b(1-e^{-\partial _{t}/a})g(x,t)
		\\
		g(x,0)=\delta (x) \\
		\lim_{|x|\rightarrow +\infty }g(x,t)=0
	\end{array}
	\right. ,  \label{res}
\end{equation}%
where $\delta (x)$ is the Dirac delta function and $e^{-\partial _{t}/a}$ is a  shift operator. For any analytic function $f:\mathbb{R}\to \mathbb{R}$ and $c\in \mathbb{R}$, it is defined as follows:
\begin{equation}\label{shi}
	e^{c\partial_{t}}f(t)\coloneqq \sum_{n=0}^{\infty}\frac{(c\partial_{t})^{n}}{n!}f(t)=f(t+c),
\end{equation}
where $\partial_{t}=\frac{\partial}{\partial t}$.
\subsection{Stable subordinator and its inverse}
A stable subordinator $\{D_{\beta}(t)\}_{t\ge0}$, $0<\beta<1$ is a non-decreasing L\'evy process
whose Laplace transform is given by $\mathbb{E}(e^{-sD_{\beta}(t)})=e^{-ts^{\beta}}$, $s>0$. Its associated Bern\v stein function is $f_{2}(s)=s^{\beta}$. Its first passage
time $\{Y_{\beta}(t)\}_{t\ge0}$ which is defined as $Y_{\beta}(t)\coloneqq\inf\{x\ge 0:D_{\beta}(x)>t\}$ is called the inverse stable subordinator.

Let $h_{\beta}(t,x)$ be the density of $\{Y_{\beta}(t)\}_{t\ge0}$. Its Laplace transform is given by
\begin{equation}\label{lap}
	\tilde{h}_{\beta}(s,x)=s^{\beta-1}e^{-xs^{\beta}}.
\end{equation}
The mean and variance of $\{Y_{\beta}(t)\}_{t\ge0}$ are given by (see Leonenko {\it et al.} (2014))
\begin{align}\label{meani}
	\mathbb{E}\left(Y_{\beta}(t)\right)&=\frac{t^{\beta}}{\Gamma(\beta+1)},\\
	\operatorname{ Var}\left(Y_{\beta}(t)\right)&=\Big(\frac{2}{\Gamma(2\beta+1)}-\frac{1}{\Gamma^{2}(\beta+1)}\Big)t^{2\beta}\label{xswe331}.
\end{align}
For fixed $s$ and large $t$, the following asymptotic result holds for its covariance:
\begin{equation}\label{covin}
	\operatorname{Cov}\left(Y_{\beta}(s),Y_{\beta}(t)\right)\sim\frac{s^{2\beta}}{\Gamma(2\beta+1)}.
\end{equation}
\subsection{Fractional negative binomial processes}
The space-time fractional negative binomial process is defined as (see Kataria and Khandakar (2022c)) 
\begin{equation*}
	\mathcal{N}^{\alpha}_{\beta}(t)\coloneqq N^{\alpha,\beta}(Z(t)),\ t\ge0,
\end{equation*}
 where the STFPP $\{N^{\alpha,\beta}(t)\}_{t\ge0}$ is independent of gamma subordinator. 
For $\alpha=1$, it reduces to the time fractional negative binomial process (see Vellaisamy and Maheshwari (2018b)) and for $\beta=1$, it reduces to the space fractional negative binomial process (see Beghin and Vellaisamy (2018)). The pgf of time fractional negative binomial process is given by (see Kataria and Khandakar (2022c))
\begin{equation}\label{pgftfnbp}
\mathbb{E}\left(u^{\mathcal{N}^{1}_{\beta}(t)}\right)=\sum_{k=0}^{\infty}\frac{(-\lambda)^{k}(1-u)^{k}}{\Gamma(k\beta+1)}\frac{\Gamma (bt+k\beta)}{a^{k\beta}\Gamma(bt)},\ |u|\le 1.
\end{equation}
Its state probability is given by (see Vellaisamy and Maheshwari (2018b), Section 4.1)
\begin{equation}\label{tfnbppm}
	\mathrm{Pr}\{\mathcal{N}^{1}_{\beta}(t)=n\}=\sum_{l=n}^{\infty}(-1)^{l+n}\binom{l}{n}\frac{\Gamma(l\beta+bt)}{\Gamma(bt)\Gamma(\beta l+1)}\left(\frac{\lambda}{a^{\beta}}\right)^{l},\ n\ge0.
\end{equation}
\subsection{LRD property}

	Let $s>0$ be fixed. The process $\{X(t)\}_{t\ge0}$ is said to exhibit the LRD property if its correlation function has the following asymptotic behaviour (see D'Ovidio and Nane (2014), Maheshwari and Vellaisamy (2016)):
	\begin{equation*}
		\operatorname{Corr}(X(s),X(t))\sim c(s)t^{-\theta},\  \text{as}\ t\rightarrow\infty,
	\end{equation*}
	for some $c(s)>0$ and $\theta\in(0,1)$. If $\theta\in(1,2)$ then it is said to possesses the short-range dependence property.

\section{Non-homogeneous GCP and its fractional variants}\label{section3}
In this section, we aim to introduce and study a non-homogeneous version of the GFCP, namely, the non-homogeneous generalized fractional counting process. We start by giving some additional results related to the GCP and its fractional variant, the GFCP.

	Let $Z_{1}$ denote the time of first occurrence of GCP event in $[0,t]$. Then, 
\begin{align*}
	\mathrm{Pr}\{Z_{1}\le x|M(t)=1\}&=\frac{\mathrm{Pr}\{\text{exactly one occurrence in} \ [0,x], \text{no occurrence in} \ (x,t]\}}{\mathrm{Pr}\{M(t)=1\}}\\
	&=\frac{\mathrm{Pr}\{M(x)=1, M(t-x)=0\}}{\mathrm{Pr}\{M(t)=1\}}\\
	&=\frac{\lambda_{1}xe^{-\Lambda x}e^{-\Lambda (t-x)}}{\lambda_{1}te^{-\Lambda t}},\ (\text{using} \ \eqref{p(n,t)})\\
	&=\frac{x}{t},\ \ 0\le x\le t.
\end{align*}
So, the distribution of $Z_{1}$ given that exactly one GCP event has occoured in $[0,t]$ is uniform in  $[0,t]$.

Di Crescenzo {\it et al.} (2016) showed that
\begin{equation}\label{mtb}
	M^{\beta}(t)\stackrel{d}{=}M(T_{2\beta}(t)),\ t>0,
\end{equation}
where the GCP is independent of the random process $\{T_{2\beta}(t)\}_{t>0}$ whose distribution is given by the folded solution of the following Cauchy problem (see Orsingher and Beghin (2004)): 
\begin{equation}\label{diff}
	\frac{\mathrm{d}^{2\beta}}{\mathrm{d}t^{2\beta}}u(x,t)=\frac{\partial^{2}}{\partial x^{2}}u(x,t),\ t>0, \ x\in\mathbb{R},
\end{equation}
with $u(x,0)=\delta(x)$ for $0<\beta\le 1$ and $\frac{\partial^{}}{\partial t}u(x,0)=0$ for $1/2<\beta\le1$. The solution $u_{2\beta}(x,t)$ of \eqref{diff} is given by (see Beghin and Orsingher (2009)):
\begin{equation}\label{u2b}
	u_{2\beta}(x,t)=\frac{1}{2t^{\beta}}W_{-\beta, 1-\beta}\left(-\frac{|x|}{t^{\beta}}\right),\ t>0, \ x\in \mathbb{R},
\end{equation}
where $W_{\nu,\gamma}(\cdot)$ is the Wright function defined as follows:
\begin{equation*}
	W_{\nu,\gamma}(x)=\sum_{k=0}^{\infty}\frac{x^{k}}{k!\Gamma(k\nu+\gamma)},\ \nu>-1,\ \gamma >0,\ x\in \mathbb{R}.
\end{equation*}
Let 
\begin{equation}\label{ub}
	\bar{u}_{2\beta}(x,t)=\begin{cases*}
		2u_{2\beta}(x,t),\ x>0, \\
		0,\ x<0,
	\end{cases*}
\end{equation}
be the folded solution to \eqref{diff}. The following result gives an alternate version of the pmf of GFCP in terms of derivatives of Mittag-Leffler function.

\begin{proposition}\label{alternate}
	The pmf of GFCP is given by
	\begin{equation*}
		p^{\beta}(n,t)=\sum_{\Omega(k,n)}\prod_{j=1}^{k}\frac{\lambda_{j}^{x_{j}}}{x_{j}!}\frac{1}{\Lambda^{z_{k}}}\frac{\mathrm{d}^{z_{k}}}{\mathrm{d}s^{z_{k}}}\left[s^{z_{k}-1}E_{\beta,1}\left(-\frac{\Lambda t^{\beta}}{s}\right)\right]_{s=1},		
	\end{equation*}
where $z_{k}=x_{1}+x_{2}+\cdots+x_{k}$.
\end{proposition}
\begin{proof}
	From \eqref{mtb} and \eqref{ub}, we have
	\begin{equation}\label{hdf}
		p^{\beta}(n,t)=\int_{0}^{\infty}p(n,x)\bar{u}_{2\beta}(x,t)\mathrm{d}x.
	\end{equation}
	On using \eqref{p(n,t)} and \eqref{u2b} in \eqref{hdf}, we get
	\begin{align*}
		p^{\beta}(n,t)&=\int_{0}^{\infty}\sum_{\Omega(k,n)}\prod_{j=1}^{k}\frac{\lambda_{j}^{x_{j}}}{x_{j}!}e^{-\Lambda x}x^{z_{k}}t^{-\beta}W_{-\beta,1-\beta}(-xt^{-\beta})\mathrm{d}x\\
		&=\int_{0}^{\infty}\sum_{\Omega(k,n)}\prod_{j=1}^{k}\frac{\lambda_{j}^{x_{j}}}{x_{j}!}e^{-\Lambda x}x^{z_{k}}t^{-\beta}\sum_{m=0}^{\infty}\frac{(-xt^{-\beta})^{m}}{m!\Gamma(-\beta m+1-\beta)}\mathrm{d}x\\
		&=\sum_{\Omega(k,n)}\prod_{j=1}^{k}\frac{\lambda_{j}^{x_{j}}}{x_{j}!}\sum_{m=0}^{\infty}\frac{(-1)^{m}t^{-\beta-\beta m}}{m!\Gamma(-\beta m+1-\beta)}\int_{0}^{\infty}e^{-\Lambda x}x^{m+z_{k}}\mathrm{d}x\\
		&=\sum_{\Omega(k,n)}\prod_{j=1}^{k}\frac{\lambda_{j}^{x_{j}}}{x_{j}!}\sum_{m=0}^{\infty}\frac{(-1)^{m}t^{-\beta-\beta m}}{m!\Gamma(-\beta m+1-\beta)}\frac{\Gamma(m+z_{k}+1)}{\Lambda^{m+z_{k}+1}}\\
		&=\sum_{\Omega(k,n)}\prod_{j=1}^{k}\frac{\lambda_{j}^{x_{j}}}{x_{j}!}\frac{t^{-\beta}}{\Lambda^{z_{k}+1}}\sum_{m=0}^{\infty}\frac{(-t^{\beta}\Lambda)^{-m}}{\Gamma(-\beta m+1-\beta)}(m+z_{k})(m+z_{k}-1)\cdots (m+1)\\
		&=\sum_{\Omega(k,n)}\prod_{j=1}^{k}\frac{\lambda_{j}^{x_{j}}}{x_{j}!}\frac{t^{-\beta}}{\Lambda^{z_{k}+1}}\frac{\mathrm{d}^{z_{k}}}{\mathrm{d}s^{z_{k}}}\bigg[\sum_{m=0}^{\infty}s^{m+z_{k}}\frac{(-t^{\beta}\Lambda)^{-m}}{\Gamma(-\beta m+1-\beta)}\bigg]_{s=1}\\
		&=\sum_{\Omega(k,n)}\prod_{j=1}^{k}\frac{\lambda_{j}^{x_{j}}}{x_{j}!}\frac{\mathrm{d}^{z_{k}}}{\mathrm{d}s^{z_{k}}}\left[s^{z_{k}}E_{-\beta,1-\beta}\left(-\frac{s}{\Lambda t^{\beta}}\right)\right]_{s=1}.
	\end{align*}
	Finally, the proof follows on using \eqref{mittag}.
\end{proof}
\begin{remark}
	On substituting $k=1$ in Proposition \ref{alternate}, we get an equivalent expression for the pmf of TFPP (see Beghin and Orsingher (2009), p. 1801). For $\lambda_{j}=\lambda$, $1\le j\le k$, we get the following alternate version of the pmf of time fractional Poisson process of order $k$:
	\begin{equation*}
		p^{\beta}(n,t)\big|_{\lambda_{j}=\lambda}=\sum_{\Omega(k,n)}\prod_{j=1}^{k}\frac{(1/k)^{x_{j}}}{x_{j}!}\frac{\mathrm{d}^{z_{k}}}{\mathrm{d}s^{z_{k}}}\left[s^{z_{k}-1}E_{\beta,1}\left(-\frac{k\lambda t^{\beta}}{s}\right)\right]_{s=1}.		
	\end{equation*}
	Similarly, we can obtain alternate expressions for the pmfs of other particular and limiting cases of the GFCP, for example, the fractional P\'olya-Aeppli process of order $k$, the fractional Bell-Touchard process, the fractional Poisson-logarthmic process, the convoluted fractional Poisson process, {\it etc}.
\end{remark}
\begin{remark}
	If $X_{1}, X_{2},\dots, X_{n}$ are iid random variables with distribution $F(x)=\mathrm{Pr}\{X<x\}$, then the following holds:
\small	\begin{align*}
		\mathrm{Pr}\Big\{\max_{1\le j\le M^{\beta}(t) }X_{j}<x\Big\}&=\sum_{n=0}^{\infty}(\mathrm{Pr}\{X<x\})^{n}\mathrm{Pr}\{M^{\beta}(t)=n\}\\
		&=\int_{0}^{\infty}G(F(x),s)\mathrm{Pr}\{T_{2\beta}(t)\in \mathrm{d}s\},\ (\text{using}\ \eqref{pgfmt}\ \text{and}\  \eqref{mtb})\\
		&=\int_{0}^{\infty}\exp\bigg(-s\sum_{j=1}^{k}\lambda_{j}(1-F^{j}(x))\bigg)t^{-\beta}W_{-\beta,1-\beta}(-st^{-\beta})\mathrm{d}s,\ (\text{using}\ \eqref{ub}) \\
		&=E_{\beta,1}\bigg(-\sum_{j=1}^{k}\lambda_{j}(1-F^{j}(x))t^{\beta}\bigg),
	\end{align*}
\normalsize	where in the last step we have used Eq. (2.13) of Beghin and Orsingher (2010). Similarly, it can be shown that 
	\begin{equation*}
		\mathrm{Pr}\Big\{\min_{1\le j\le M^{\beta}(t) }X_{j}>x\Big\}	=E_{\beta,1}\bigg(-\sum_{j=1}^{k}\lambda_{j}\left(1-(1-F(x))^{j}\right)t^{\beta}\bigg).
	\end{equation*}
\end{remark}

Next, we introduce a non-homogeneous version of the GCP. Let $\{\mathcal{M}(t)\}_{t\ge0}$ be a counting process with deterministic and
time-dependent intensity functions $\lambda_{j}(t):[0,\infty)\to [0,\infty)$, $1\le j \le k$ such that $\mathcal{M}(0)=0$. We call it the non-homogeneous generalized counting process (NGCP) if it has independent increments and its transition probabilities are given by 
		\begin{equation*}
			\mathrm{Pr}\{\mathcal{M}(t+h)=n|\mathcal{M}(t)=m\}=\begin{cases*}
				1-\sum_{j=1}^{k}\lambda_{j}(t) h+o(h),\ n=m,\\
				\lambda_{j}(t)h+o(h), \ n=m+j, \ j=1,2,\dots,k,\\
				o(h), \ n>m+k,
			\end{cases*}
		\end{equation*}
		where $o(h)\to 0$ as $h\to 0$. Thus, the system of differential equations that governs its state probabilities $q_{n}(t)=\mathrm{Pr}\{\mathcal{M}(t)=n\}$ are given by
\begin{equation}\label{model}
	\frac{\mathrm{d}}{\mathrm{d}t}q_{n}(t)=- \sum_{j=1}^{k}\lambda_{j}(t)q_{n}(t)+	\sum_{j=1}^{n\wedge k}\lambda_{j}(t)q_{n-j}(t), \ n\ge0,
\end{equation}
with the initial conditions
\begin{equation*}
	q_{n}(0)=\begin{cases}
		1, \ n=0,\\
		0, \ n\ge1.
	\end{cases}
\end{equation*}

For $0\le s<t$ and $j\in \{1,2,\dots,k\}$, let $\Lambda_{j}(t)=\int_{0}^{t}\lambda_{j}(u)\mathrm{d}u<\infty$ be the cumulative rate function and $\Lambda_{j}(s,t)=\int_{s}^{t}\lambda_{j}(u)\mathrm{d}u=\Lambda_{j}(t)-\Lambda_{j}(s)$. 
On using \eqref{model}, the pgf of NGCP can be obtained as
\begin{equation*}
	\mathbb{E}\left(u^{\mathcal{M}(t)}\right)=\exp\bigg(\sum_{j=1}^{k}\Lambda_{j}(t)(u^{j}-1)\bigg).
\end{equation*}
 \begin{remark}
	The NGCP reduces to the GCP for	constant rates, that is, $\lambda_{j}(t)=\lambda_{j}$,  $t\ge0$, $1\le j\le k$.
For $k=1$, it reduces to the non-homogeneous Poisson process. For $\lambda_{j}(t)=\lambda(t)$ and $\lambda_{j}(t)=(1-\rho)\rho^{j-1}\lambda(t)/(1-\rho^{k})$, $1\le j\le k$  it reduces to the  non-homogeneous Poisson process of order $k$ and the non-homogeneous P\'olya-Aeppli process of order $k$, respectively (see Kadankova {\it et al.} (2021)). By letting $k\to \infty$ in (\ref{model}) and taking $\lambda_{j}(t)= \lambda(t)(1-\rho)\rho^{j-1}$, $j\ge 1$ the system (\ref{model})
reduces to the governing system of differential equations for the state probabilities
of non-homogeneous P\'olya-Aeppli process (see Chukova and Minkova (2019)).  Chukova and Minkova (2019) showed an application of  non-homogeneous P\'olya-Aeppli process in risk theory, along the similar lines it can be shown that the NGCP has application in risk theory.
\end{remark}
The increment process of NGCP for $v\ge 0$ is defined as 
\begin{equation*}
	I(t,v)\coloneqq\mathcal{M}\left(t+v\right)-\mathcal{M}(v), \  t\ge0.
\end{equation*}
Its marginal distribution $	q_{n}(t,v)=\mathrm{Pr}\{I(t,v)=n\}$ is given by
\begin{equation*}
	q_{n}(t,v)=\sum_{\Omega(k,n)}\prod_{j=1}^{k}\frac{(\Lambda_{j}(v,t+v))^{x_{j}}}{x_{j}!}e^{-\sum_{j=1}^{k}\Lambda_{j}(v,t+v)}, \ n\ge0
\end{equation*}
and its characteristic function is of the following form:
\begin{equation}\label{itvxi}
	\hat{q}_\xi(t,v)=	\exp\bigg(\sum_{j=1}^{k}\Lambda_{j}(v, t+v)(e^{\omega \xi j}-1)\bigg),\ \omega=\sqrt{-1},\ \xi\in \mathbb{R}.
\end{equation}

\subsection{NGCP time-changed by inverse stable subordinator}
For $0<\beta<1$, let $\{\mathcal{M}_{\beta}(t)\}_{t\ge0}$ be a time-changed NGCP defined as follows:
\begin{equation*}
	\mathcal{M}_{\beta}(t)\coloneqq\mathcal{M}(Y_{\beta}(t)),\  t\ge0,
\end{equation*}
where $\{\mathcal{M}(t)\}_{t\ge0}$ is independent of the inverse stable subordinator $\{Y_{\beta}(t)\}_{t\ge0}$. We call it the non-homogeneous generalized fractional counting process (NGFCP).
Its increment process is defined as
\begin{equation*}
	I_{\beta}(t,v)\coloneqq\mathcal{M}\left(Y_{\beta}(t)+v\right)-\mathcal{M}(v),\ t\ge0,
\end{equation*} 
whose  marginal distributions, that is,  $q^{\beta}_{n}(t,v)=\mathrm{Pr}\{I_{\beta}(t,v)=n\}$, $n\ge0$ are given by
\begin{align}\label{qna}
	q^{\beta}_{n}(t,v)&=\int_{0}^{\infty}q_{n}(u,v)h_{\beta}(t,u)\mathrm{d}u\\
	&=\int_{0}^{\infty}\sum_{\Omega(k,n)}\prod_{j=1}^{k}\frac{(\Lambda_{j}(v,u+v))^{x_{j}}}{x_{j}!}e^{-\sum_{j=1}^{k}\Lambda_{j}(v,u+v)}h_{\beta}(t,u)\mathrm{d}u.\nonumber
\end{align}
Here, $h_{\beta}(t,u)$ is the density of $\{Y_{\beta}(t)\}_{t\ge0}$. So, the marginal distributions of NGFCP are given by
\begin{align}\label{ngfcp}
q^{\beta}_{n}(t)=q^{\beta}_{n}(t,0)&=\int_{0}^{\infty}q_{n}(u,0)h_{\beta}(t,u)\mathrm{d}u\nonumber\\
&=\int_{0}^{\infty}\sum_{\Omega(k,n)}\prod_{j=1}^{k}\frac{(\Lambda_{j}(u))^{x_{j}}}{x_{j}!}e^{-\sum_{j=1}^{k}\Lambda_{j}(u)}h_{\beta}(t,u)\mathrm{d}u,
\end{align}
where we have used $\Lambda_{j}(0,u)=\Lambda_{j}(u)$.

 Let $\delta_{n,0}$ denote the
	Kronecker delta function.
\begin{theorem}\label{thm1}
The pmf of increment process of NGFCP satisfies the following system of fractional differential-integral equations:
	\begin{equation}\label{non}
		\frac{\mathrm{d}^{\beta}}{\mathrm{d}t^{\beta}}q^{\beta}_{n}(t,v)=\int_{0}^{\infty}\bigg(-\sum_{j=1}^{k}\lambda_{j}(u+v)q_{n}(u,v)+\sum_{j=1}^{n\wedge k}\lambda_{j}(u+v)q_{n-j}(u,v)\bigg)h_{\beta}(t,u)\mathrm{d}u,\ n\ge 0,
	\end{equation}
	with the initial condition $q^{\beta}_{n}(0,v)=\delta_{n,0}$.
\end{theorem}
\begin{proof}
The characteristic function of $I_{\beta}(t,v)$ can be written as
\begin{equation}\label{xi}
	\hat{q}_\xi^\beta(t,v)=\int_{0}^{\infty}\hat{q}_{\xi}(u,v)h_{\beta}(t,u)\mathrm{d}u,
\end{equation}
where $\hat{q}_{\xi}(u,v)$ is the characteristic function of $I(t,v)$ given in (\ref{itvxi}).
On taking the  Laplace transform in \eqref{xi}, we get
	\begin{align*}
		\tilde{\hat{q}}_\xi^\beta(s,v)&=\int_{0}^{\infty}\hat{q}_{\xi}(u,v)\tilde{h}_{\beta}(s,u)\mathrm{d}u\\
		&=s^{\beta-1}\int_{0}^{\infty}\exp\bigg(\sum_{j=1}^{k}\Lambda_{j}(v, u+v)(e^{\omega \xi j}-1)\bigg)e^{-us^{\beta}}\mathrm{d}u,\ \ (\text{using (\ref{lap}) and (\ref{itvxi})})\\
	&=s^{\beta-1}\bigg\{\bigg[-\frac{e^{-us^{\beta}}}{s^{\beta}}\exp\bigg(\sum_{j=1}^{k}\Lambda_{j}(v, u+v)(e^{\omega \xi j}-1)\bigg)\bigg]_{0}^{\infty}\\
		&\ \ \ +\frac{1}{s^{\beta}}\int_{0}^{\infty}\sum_{j=1}^{k}\lambda_{j}(u+v)(e^{\omega \xi j}-1)\exp\bigg(\sum_{j=1}^{k}\Lambda_{j}(v, u+v)(e^{\omega \xi j}-1)\bigg)e^{-us^{\beta}}\mathrm{d}u\bigg\}.
	\end{align*}
	Thus,
	\begin{equation*}
		s^{\beta}\tilde{\hat{q}}_\xi^\beta(s,v)-s^{\beta-1}=\int_{0}^{\infty}\sum_{j=1}^{k}\lambda_{j}(u+v)(e^{\omega \xi j}-1)\exp\bigg(\sum_{j=1}^{k}\Lambda_{j}(v,u+v)(e^{\omega \xi j}-1)\bigg)s^{\beta-1}e^{-us^{\beta}}\mathrm{d}u.
	\end{equation*}
On taking the inverse Laplace transform, and using (\ref{lc}) and \eqref{lap}, we get
	\begin{align*}
		\frac{\mathrm{d}^{\beta}}{\mathrm{d}t^{\beta}}\hat{q}_\xi^\beta(t,v)&=\int_{0}^{\infty}\sum_{j=1}^{k}\lambda_{j}(u+v)(e^{\omega \xi j}-1)\exp\bigg(\sum_{j=1}^{k}\Lambda_{j}(v,u+v)(e^{\omega \xi j}-1)\bigg)h_{\beta}(t,u)\mathrm{d}u\\
		&=\int_{0}^{\infty}\sum_{j=1}^{k}\lambda_{j}(u+v)(e^{\omega \xi j}-1)	\hat{p}_\xi(u,v)h_{\beta}(t,u)\mathrm{d}u,
	\end{align*} 
where we have used (\ref{itvxi}) in the last step.	Finally, on using the inversion formula for the characteristic function, we get the required result.
\end{proof}

On substituting $v=0$ in Theorem \ref{thm1}, we get the system of differential equation that governs the pmf of NGFCP.
\begin{corollary}
	The pmf $q^{\beta}_{n}(t)$ of NGFCP satisfies the following system of fractional differential-integral equations:
	\begin{equation}\label{cor1}
	\frac{\mathrm{d}^{\beta}}{\mathrm{d}t^{\beta}}q^{\beta}_{n}(t)=\int_{0}^{\infty}\bigg(-\sum_{j=1}^{k}\lambda_{j}(u)q_{n}(u)+\sum_{j=1}^{n\wedge k}\lambda_{j}(u)q_{n-j}(u)\bigg)h_{\beta}(t,u)\mathrm{d}u,\ n\ge 0,
	\end{equation}
	with the initial condition $q^{\beta}_{n}(0)=\delta_{n,0}$.
\end{corollary}

\begin{remark}
	 On choosing constant rates, that is, $\lambda_{j}(t)=\lambda_{j}$, for $1\le j \le k$ in \eqref{cor1}, we get
	\begin{align*}
	\frac{\mathrm{d}^{\beta}}{\mathrm{d}t^{\beta}}p^{\beta}(n,t)&=\int_{0}^{\infty}\bigg(-\sum_{j=1}^{k}\lambda_{j}p(n,u)+\sum_{j=1}^{n\wedge k}\lambda_{j}p(n-j,u)\bigg)h_{\beta}(t,u)\mathrm{d}u,
	\end{align*}
	which reduces to \eqref{cre}. Similarly, by taking constant rates in \eqref{ngfcp}, we get
	\begin{equation*}
	p^{\beta}(n,t)=\int_{0}^{\infty}\sum_{\Omega(k,n)}\prod_{j=1}^{k}\frac{(\lambda_{j}u)^{x_{j}}}{x_{j}!}e^{-\Lambda u}h_{\beta}(t,u)\mathrm{d}u,
	\end{equation*}
which coincides with the pmf of GFCP.

 On substituting $\lambda_{j}(t)=\lambda(t)$ and $\lambda_{j}(t)=(1-\rho)\rho^{j-1}\lambda(t)/(1-\rho^{k})$, $1\le j\le k$ in Theorem \ref{thm1}, we get the system of differential-integral equations that governs the pmf of corresponding increment process of the non-homogeneous fractional Poisson process of order $k$ and the non-homogeneous fractional P\'olya-Aeppli process of order $k$, respectively (see Kadankova {\it et al.} (2021), Theorem 3.3 and Theorem 4.4).
\end{remark}
\begin{remark}
	 For $\beta=1$, the Laplace transform of inverse stable subordinator reduces to $e^{-us}$ whose inverse Laplace transform is  $\delta(t-u)$, the Dirac's delta distribution. Hence, 
	\begin{equation*}
	\int_{0}^{\infty}q_{n}(u,v)\delta(t-u)\mathrm{d}u=q_{n}(t,v).
	\end{equation*}
From \eqref{qna}, it follows that $q_{n}^{1}(t,v)=q_{n}(t,v)$ and the system of differential-integral equations in \eqref{non} reduces to
	\begin{align*}
	\frac{\mathrm{d}}{\mathrm{d}t}q_{n}(t,v)&=\int_{0}^{\infty}\bigg(-\sum_{j=1}^{k}\lambda_{j}(u+v)q_{n}(u,v)+\sum_{j=1}^{n\wedge k}\lambda_{j}(u+v)q_{n-j}(u,v)\bigg)\delta(t-u)\mathrm{d}u\\
	&=-\sum_{j=1}^{k}\lambda_{j}(t+v)q_{n}(t,v)+\sum_{j=1}^{n\wedge k}\lambda_{j}(t+v)q_{n-j}(t,v).
	\end{align*}
\end{remark}

The mean, variance and covariance of NGCP are given by 
$
\mathbb{E}(\mathcal{M}(t))=\sum_{j=1}^{k}j\Lambda_{j}(t)$,
$\operatorname{Var}\left(\mathcal{M}(t)\right)=\sum_{j=1}^{k}j^{2}\Lambda_{j}(t)$ and $\operatorname{Cov}\left(\mathcal{M}(s),\mathcal{M}(t)\right)=\sum_{j=1}^{k}j^{2}\Lambda_{j}(s\wedge t)$, respectively.
\begin{proposition}\label{martingalen}
		The process $\{\mathcal{M}(t)-\sum_{j=1}^{k}j\Lambda_{j}(t)\}_{t\geq0}$ is a martingale with respect to a natural filtration $\mathscr{F}_{t}=\sigma\left(\mathcal{M}(s), 0<s\le t\right)$.
	\end{proposition}
	\begin{proof}
		Let $\bar{Q}(t)=\mathcal{M}(t)-\sum_{j=1}^{k}j\Lambda_{j}(t)$. As NGCP has independent increments, we have the following for $0<s\le t$:
		\begin{equation*}
		\mathbb{E}\left(\bar{Q}(t)-\bar{Q}(s)|\mathscr{F}_{s}\right)=	\mathbb{E}\left(\mathcal{M}(t)-\mathcal{M}(s)\big|\mathscr{F}_{s}\right)-\sum_{j=1}^{k}j\left(\Lambda_{j}(t)-\Lambda_{j}(s)\right)=0.
		\end{equation*}
		This completes the proof.
	\end{proof}

Next,  we discuss the arrival times of NGCP. Let $\mathscr{T}_{n}=\min\{t\ge 0:\mathcal{M}(t)=n\}$ be the arrival time of $n$th NGCP event. Then,
\begin{equation*}
F_{\mathscr{T}_{n}}(t)=\mathrm{Pr}\{\mathscr{T}_{n}\le t\}=\mathrm{Pr}\{\mathcal{M}(t)\ge n\}\\
=1-\sum_{x=0}^{n-1}\sum_{\Omega(k,x)}\prod_{j=1}^{k}\frac{(\Lambda_{j}(t))^{x_{j}}}{x_{j}!}e^{-\sum_{j=1}^{k}\Lambda_{j}(t)}.
\end{equation*}
We note that $F_{\mathscr{T}_{n}}(t)$ gives a distribution function if and only if $\Lambda_{j}(t)$'s satisfy the conditions given in Remark 5 of Leonenko {\it et al.} (2017).

Similarly, let $\mathscr{T}_{n}^{\beta}=\min\{t\ge 0:\mathcal{M}_{\beta}(t)=n\}$ be the arrival time of $n$th NGFCP event. On using (\ref{ngfcp}), its distribution can be written as
\begin{align*}
F_{\mathscr{T}_{n}^{\beta}}(t)
&=\sum_{x=n}^{\infty}\int_{0}^{\infty}\sum_{\Omega(k,x)}\prod_{j=1}^{k}\frac{(\Lambda_{j}(u))^{x_{j}}}{x_{j}!}e^{-\sum_{j=1}^{k}\Lambda_{j}(u)}h_{\beta}(t,u)\mathrm{d}u\\
&=\int_{0}^{\infty}\sum_{x=n}^{\infty}\sum_{\Omega(k,x)}\prod_{j=1}^{k}\frac{(\Lambda_{j}(u))^{x_{j}}}{x_{j}!}e^{-\sum_{j=1}^{k}\Lambda_{j}(u)}h_{\beta}(t,u)\mathrm{d}u\\
&=\int_{0}^{\infty}\bigg(1-\sum_{x=0}^{n-1}\sum_{\Omega(k,x)}\prod_{j=1}^{k}\frac{(\Lambda_{j}(u))^{x_{j}}}{x_{j}!}e^{-\sum_{j=1}^{k}\Lambda_{j}(u)}\bigg)h_{\beta}(t,u)\mathrm{d}u\\
&=\int_{0}^{\infty}F_{\mathscr{T}_{n}}(u)h_{\beta}(t,u)\mathrm{d}u.
\end{align*}

\subsection{NGCP time-changed by inverse subordinator}
Let $\{Y_{f}(t)\}_{t\ge0}$ be the inverse subordinator  defined in \eqref{Yf} such that $c_{1}=c_{2}=0$ in its  associated Bern\v stein function given in (\ref{fx}). We consider a time-changed process defined as follows:
\begin{equation*}
\mathcal{M}_{f}(t)\coloneqq\mathcal{M}(Y_{f}(t)), \ t\ge0, 
\end{equation*}
where the NGCP $\{\mathcal{M}(t)\}_{t\ge0}$ is independent of $\{Y_{f}(t)\}_{t\ge0}$.

For $f(s)=s^{\beta}$, $0<\beta<1$ the process $\{\mathcal{M}_{f}(t)\}_{t\ge0}$ reduces to the NGFCP. For $k=1$, it reduces to the non-homogeneous Poisson process time-changed by inverse subordinator that was introduced and studied by Buchak and Sakhno (2019).  

Let $v\ge 0$. The increment process of $\{\mathcal{M}_{f}(t)\}_{t\ge0}$ is defined as 
\begin{equation*}
I_{f}(t,v)=I(Y_{f}(t),v)=\mathcal{M}\left(Y_{f}(t)+v\right)-\mathcal{M}(v),  \ t\ge0.
\end{equation*} 
Its marginal distribution $q_n^f(t,v)=\mathrm{Pr}\left\{I_{f}(t,v)=n\right\}$, $n\ge0$ is given by
\begin{align*}
	q_n^f(t,v)&=\int_{0}^{\infty}q_n(u,v)h_{f}(t,u)\mathrm{d}u\\
	&=\int_{0}^{\infty}\sum_{\Omega(k,n)}\prod_{j=1}^{k}\frac{(\Lambda_{j}(v,u+v))^{x_{j}}}{x_{j}!}e^{-\sum_{j=1}^{k}\Lambda_{j}(v,u+v)}h_{f}(t,u)\mathrm{d}u,
\end{align*}
where $h_{f}(t,u)$ is the density of $\{Y_{f}(t)\}_{t\ge0}$, which exists under Condition {\bf I}.
So, the marginal distribution of $\{\mathcal{M}_{f}(t)\}_{t\ge0}$, that is, $q_n^f(t)=q_n^f(t,0)=\mathrm{Pr}\{\mathcal{M}_{f}(t)=n\}$  is given by
\begin{align*}
	q_n^f(t)&=\int_{0}^{\infty}q_n(u)h_{f}(t,u)\mathrm{d}u\\
	&=\int_{0}^{\infty}\sum_{\Omega(k,n)}\prod_{j=1}^{k}\frac{(\Lambda_{j}(u))^{x_{j}}}{x_{j}!}e^{-\sum_{j=1}^{k}\Lambda_{j}(u)}h_{f}(t,u)\mathrm{d}u.
\end{align*}

\begin{theorem}\label{3.3}
Let $^f\mathcal{D}_t$ be the generalized Caputo derivative with respect to the Bern\v{s}tein function $f$ defined in \eqref{fDt}, and let Condition {\bf I}  hold. Then, the pmf of $\{I_{f}(t,v)\}_{t\ge0}$ satisfies the following system of differential-integral equations:
	\begin{equation*}
	^f\mathcal{D}_tq_n^f(t,v)=\int_{0}^{\infty}\bigg(-\sum_{j=1}^{k}\lambda_{j}(u+v)q_n(u,v)+\sum_{j=1}^{n\wedge k}\lambda_{j}(u+v)q_{n-j}(u,v)\bigg)h_{f}(t,u)\mathrm{d}u,
	\end{equation*}
	with initial condition
	$q^{f}_{n}(0,v)=\delta_{n,0}$.
\end{theorem}
\begin{proof}
	The proof follows similar lines to that of Theorem \ref{thm1}. We present it here for the sake of completeness.
	
	The characteristic function of $I_{f}(t,v)$ can be written as
	\begin{equation}\label{xif}
		\hat{q}_\xi^f(t,v)=\int_{0}^{\infty}\hat{q}_{\xi}(u,v)h_{f}(t,u)\mathrm{d}u.
	\end{equation}
	On taking the  Laplace transform in \eqref{xif}, we get
	\begin{align*}
		\tilde{\hat{q}}_\xi^f(s,v)&=\int_{0}^{\infty}\hat{q}_{\xi}(u,v)\tilde{h}_{f}(s,u)\mathrm{d}u\\
		&=\frac{f(s)}{s}\int_{0}^{\infty}\exp\bigg(\sum_{j=1}^{k}\Lambda_{j}(v, u+v)(e^{\omega \xi j}-1)\bigg)e^{-uf(s)}\mathrm{d}u,\ \ (\text{using \eqref{lhft} and (\ref{itvxi})})\\
		&=\frac{f(s)}{s}\Bigg\{\left[-\frac{e^{-uf(s)}}{f(s)}\exp\bigg(\sum_{j=1}^{k}\Lambda_{j}(v, u+v)(e^{\omega \xi j}-1)\bigg)\right]_{0}^{\infty}\\
		&\ \ \ +\frac{1}{f(s)}\int_{0}^{\infty}\sum_{j=1}^{k}\lambda_{j}(u+v)(e^{\omega \xi j}-1)\exp\bigg(\sum_{j=1}^{k}\Lambda_{j}(v, u+v)(e^{\omega \xi j}-1)\bigg)e^{-uf(s)}\mathrm{d}u\Bigg\}.
	\end{align*}
	Thus,
\small	\begin{equation*}
		f(s)\tilde{\hat{q}}_\xi^f(s,v)-\frac{f(s)}{s}=\int_{0}^{\infty}\sum_{j=1}^{k}\lambda_{j}(u+v)(e^{\omega \xi j}-1)\exp\bigg(\sum_{j=1}^{k}\Lambda_{j}(v,u+v)(e^{\omega \xi j}-1)\bigg)\frac{f(s)}{s}e^{-uf(s)}\mathrm{d}u.
	\end{equation*}
\normalsize	On taking the inverse Laplace transform, and using \eqref{lhft} and \eqref{LfDt}, we get
	\begin{align*}
		^f\mathcal{D}_t\hat{q}_\xi^f(t,v)&=\int_{0}^{\infty}\sum_{j=1}^{k}\lambda_{j}(u+v)(e^{\omega \xi j}-1)\exp\bigg(\sum_{j=1}^{k}\Lambda_{j}(v,u+v)(e^{\omega \xi j}-1)\bigg)h_{f}(t,u)\mathrm{d}u\\
		&=\int_{0}^{\infty}\sum_{j=1}^{k}\lambda_{j}(u+v)(e^{\omega \xi j}-1)	\hat{q}_\xi(u,v)h_{f}(t,u)\mathrm{d}u,
	\end{align*} 
	where we have used (\ref{itvxi}) in the last step.	Now, the desired result follows on using the inversion formula for the characteristic function.
\end{proof}

On substituting $v=0$ in Theorem \ref{3.3}, we get the system of differential-integral equations that governs the pmf of $\{\mathcal{M}_{f}(t)\}_{t\ge0}$.
\begin{corollary}
	Let Condition {\bf I}  hold. Then, the pmf of $\{\mathcal{M}_{f}(t)\}_{t\ge0}$ satisfy the following system of differential-integral equations:
	\begin{equation}\label{fdtlu}
^f\mathcal{D}_tq_n^f(t)=\int_{0}^{\infty}\bigg(-\sum_{j=1}^{k}\lambda_{j}(u)q_n(u)+\sum_{j=1}^{n\wedge k}\lambda_{j}(u)q_{n-j}(u)\bigg)h_{f}(t,u)\mathrm{d}u,\ n\ge0,
\end{equation}
	with initial condition $q^{f}_{n}(0)=\delta_{n,0}$.
\end{corollary}

		On choosing constant rates, that is, $\lambda_{j}(t)=\lambda_{j}$ for $1\le j \le k$ in \eqref{fdtlu}, we get 
		\begin{equation}\label{gfd}
			^f\mathcal{D}_tp_n^f(t)=-\Lambda p_n^f(t)+\sum_{j=1}^{n\wedge k}\lambda_{j}p_n^f(t),
		\end{equation}
	where 	$p_n^f(t)=\mathrm{Pr}\left\{M\left(Y_f(t)\right)=n\right\}$ are the state probabilities of $\{M\left(Y_{f}(t)\right)\}_{t\ge0}$, a process studied in  Kataria and Khandakar (2022a).

Alternatively, the system of differential equations \eqref{gfd} can be obtained as follows:

Note that	
	\begin{equation}\label{ghf}
p_n^f(t)=\int_{0}^{\infty}p(n,u)h_{f}(t,u)\mathrm{d}u, \ n\ge0.
	\end{equation}	
On taking the generalized Riemann-Liouville derivative $^f\mathscr{D}_t$
and using the following results (see Toaldo (2015), Theorem 4.1):
\begin{equation*}
	^f\mathscr{D}_th_{f}(t,u)=-\frac{\partial}{\partial u}h_{f}(t,u), \ h_{f}(t,0)=\nu(t),\ h_{f}(0,u)=\delta(u)
\end{equation*}
in \eqref{ghf}, we get
	\begin{align}\label{hgfd}
		^f\mathscr{D}_t	p_n^f(t)&=-\int_{0}^{\infty}p(n,u)\frac{\partial}{\partial u}h_{f}(t,u)\mathrm{d}u\nonumber\\
		&=-\left[p(n,u)h_{f}(t,u)\right]_{u=0}^{\infty}+\int_{0}^{\infty}h_{f}(t,u)\frac{\mathrm{d}}{\mathrm{d}u}p(n,u)\mathrm{d}u\nonumber\\
		&=p(n,0)h_{f}(t,0)+\int_{0}^{\infty}h_{f}(t,u)\Big(-\Lambda p(n,u)+	\sum_{j=1}^{n\wedge k}\lambda_{j}p(n-j,u)\Big)\mathrm{d}u\nonumber\\
		&=p(n,0)\nu(t)-\Lambda p_n^f(t)+\sum_{j=1}^{n\wedge k}\lambda_{j}p_{n-j}^f(t),
	\end{align}
where in the penultimate step we have used \eqref{cre} with $\beta=1$. Also,
\begin{equation*}
	p_n^f(0)=\int_{0}^{\infty}p(n,u)h_{f}(0,u)\mathrm{d}u=\int_{0}^{\infty}p(n,u)\delta(u)\mathrm{d}u=p(n,0).
\end{equation*}
From \eqref{relation}, we get
	\begin{equation*}
^f\mathcal{D}_tp_n^f(t)=^f\mathscr{D}_tp_n^f(t)-\nu(t)p_n^f(0),
	\end{equation*}
which on using in \eqref{hgfd} gives the required result.
\section{GCP time-changed by multistable subordinator}\label{section4}

Let $\{H(t)\}_{t\ge0}$ be a multistable subordinator with stability index $\beta(t)\in (0,1)$. Here, we consider the following time-changed process:
\begin{equation}\label{gsmcp}
X(t)\coloneqq M(H(t)), \ t\ge0,
\end{equation}
where the GCP $\{M(t)\}_{t\ge0}$ is independent of $\{H(t)\}_{t\ge0}$.
We call it the generalized space-multifractional counting process (GSMCP). 

For $k=1$, the GCP reduces to the Poisson process. So, the GSMCP reduces to a special process introduced and studied by Beghin and Riccuti (2019), namely, the space-multifractional Poisson process. Using Proposition 2.1 of Beghin and Riccuti (2019), it follows that the GSMCP is a non-homogeneous subordinator.

Next, we obtain the L\'evy measure of GSMCP. Let $\nu_{t}(\cdot)$  and $\nu_{t}^{*}(\cdot)$ denote the L\'evy measure of $\{H(t)\}_{t\ge0}$ and $\{X(t)\}_{t\ge0}$, respectively. Again, it follows from Eq. (2.1) of  Beghin and Riccuti (2019) that 
\begin{align}\label{levymeasure}
\nu_{t}^{*}(\mathrm{d}x)&=\int_{0}^{\infty}\mathrm{Pr}\{M(s)\in \mathrm{d}x\}\nu_{t}(\mathrm{d}s)\nonumber\\
&=\sum_{n=1}^{\infty}\sum_{\Omega(k,n)}\prod_{j=1}^{k}\frac{\lambda_{j}^{x_{j}}}{x_{j}!}\delta_{n}(\mathrm{d}x)\frac{\beta(t)}{\Gamma(1-\beta(t))}\int_{0}^{\infty}e^{-\Lambda s}s^{z_{k}-\beta(t)-1}\mathrm{d}s,\ \ (\text{using} \ \eqref{nut} \ \text{and}\ \eqref{p(n,t)})\nonumber\\
&=\sum_{n=1}^{\infty}\sum_{\Omega(k,n)}\prod_{j=1}^{k}\frac{\lambda_{j}^{x_{j}}}{x_{j}!}\frac{\beta(t)}{\Gamma(1-\beta(t))}\frac{\Gamma(z_{k}-\beta(t))}{\Lambda^{z_{k}-\beta(t)}}\delta_{n}(\mathrm{d}x)\nonumber\\
&=\sum_{n=1}^{\infty}\sum_{\Omega(k,n)}\prod_{j=1}^{k}\frac{\lambda_{j}^{x_{j}}}{x_{j}!}\frac{z_{k}!(-1)^{z_{k}+1}}{\Lambda^{z_{k}-\beta(t)}}\binom{\beta(t)}{z_{k}}\delta_{n}(\mathrm{d}x),
\end{align}
where $\delta_{n}(\cdot)$ is the Dirac delta centered at $n$.
\begin{remark}
	For $k=1$ and $\Lambda=\lambda_{1}=\lambda$ (say), the L\'evy measure (\ref{levymeasure}) reduces to
	\begin{equation*}
	\nu_{t}^{*}(\mathrm{d}x)\big|_{k=1}=\lambda^{\beta(t)}\sum_{n=1}^{\infty}(-1)^{n+1}\binom{\beta(t)}{n}\delta_{n}(\mathrm{d}x),
	\end{equation*}
	which agrees with the L\'evy measure of space-multifractional Poisson process (see Beghin and Riccuti (2019), Eq. (3.1)).
\end{remark}
On using \eqref{bernstein} and \eqref{levymeasure}, the Bern\v{s}tein function $f^{*}(u,t)$ of GSMCP can be obtained  as follows:
\begin{align*}
	f^{*}(u,t)&=\int_{0}^{\infty}(1-e^{-ux})\nu_{t}^{*}(\mathrm{d}x)\\
	&=\sum_{n=1}^{\infty}\sum_{\Omega(k,n)}\prod_{j=1}^{k}\frac{\lambda_{j}^{x_{j}}}{x_{j}!}\frac{z_{k}!(-1)^{z_{k}+1}}{\Lambda^{z_{k}-\beta(t)}}\binom{\beta(t)}{z_{k}}\int_{0}^{\infty}(1-e^{-ux})\delta_{n}(\mathrm{d}x)\\
	&=\sum_{n=1}^{\infty}\sum_{\Omega(k,n)}\prod_{j=1}^{k}\frac{\lambda_{j}^{x_{j}}}{x_{j}!}\frac{z_{k}!(-1)^{z_{k}+1}}{\Lambda^{z_{k}-\beta(t)}}\binom{\beta(t)}{z_{k}}(1-e^{-un}).
\end{align*}

The next result gives the distribution of first passage times of GSMCP. 
\begin{proposition}
	Let $T_{m}$ be the time of the first upcrossing of the level $m$, {\it i.e.},
	\begin{equation*}
	T_{m}=\inf\{s\ge0:X(s)\ge m\}.
	\end{equation*}
	Then,
	\begin{equation}\label{tm}
	\mathrm{Pr}\{T_{m}>t\}=	\sum_{n=0}^{m-1}\sum_{\Omega(k,n)}\prod_{j=1}^{k}\frac{(-\lambda_{j})^{x_{j}}}{x_{j}!}\frac{\mathrm{d}^{z_{k}}}{\mathrm{d}\Lambda^{z_{k}}}e^{-\int_{0}^{t}\Lambda^{\alpha(\tau)}\mathrm{d}\tau}.
	\end{equation}
\end{proposition}
\begin{proof}
On using $\mathrm{Pr}\{T_{m}>t\}=\mathrm{Pr}\{X(t)<m\}$, we get
	\begin{align*}
	\mathrm{Pr}\{T_{m}>t\}&=\sum_{n=0}^{m-1}\mathrm{Pr}\{X(t)=n\}\\
	&=\sum_{n=0}^{m-1}\int_{0}^{\infty}\mathrm{Pr}\{M(s)=n\}\mathrm{Pr}\{H(t)\in \mathrm{d}s\},\ (\text{using} \ \eqref{gsmcp})\\
	&=\sum_{n=0}^{m-1}\sum_{\Omega(k,n)}\prod_{j=1}^{k}\frac{\lambda_{j}^{x_{j}}}{x_{j}!}\int_{0}^{\infty}e^{-\Lambda s}s^{z_{k}}\mathrm{Pr}\{H(t)\in \mathrm{d}s\},\ (\text{using} \ \eqref{p(n,t)})\\
	&=\sum_{n=0}^{m-1}\sum_{\Omega(k,n)}\prod_{j=1}^{k}\frac{\lambda_{j}^{x_{j}}}{x_{j}!}(-1)^{z_{k}}\frac{\mathrm{d}^{z_{k}}}{\mathrm{d}\Lambda^{z_{k}}}\int_{0}^{\infty}e^{-\Lambda s}\mathrm{Pr}\{H(t)\in \mathrm{d}s\}\\
	&=\sum_{n=0}^{m-1}\sum_{\Omega(k,n)}\prod_{j=1}^{k}\frac{(-\lambda_{j})^{x_{j}}}{x_{j}!}\frac{\mathrm{d}^{z_{k}}}{\mathrm{d}\Lambda^{z_{k}}}e^{-\int_{0}^{t}\Lambda^{\beta(\tau)}\mathrm{d}\tau}.
	\end{align*}
	This completes the proof.
\end{proof}
\begin{remark}
	On substituting $k=1$ in \eqref{tm}, we get the corresponding result for space-multifractional Poisson process (see Beghin and Riccuti (2019), Section 3.3).
\end{remark}
\subsection{Generalized space fractional counting process}

The multistable subordinator $\{H(t)\}_{t\ge0}$ reduces to the stable subordinator $\{D_{\beta}(t)\}_{t\ge0}$ when the stability index is a constant, that is, $\beta(t)=\beta\in (0,1)$ for all $t\ge0$. In this case, the GSMCP reduces to the following time-changed process: 
\begin{equation}\label{md}
	Y(t)=M(D_{\beta}(t)),  \ t\ge0,
\end{equation} 
where the GCP is independent of the stable subordinator. We call the process $\{Y(t)\}_{t\ge0}$ as the generalized space fractional counting process (GSFCP).

When $k=1$, the GSFCP reduces to the space fractional Poisson process (see Orsingher and Polito (2012)). For $\lambda_{j}=\lambda$, $j=1,2,\dots,k$ it reduces to the space fractional Poisson process of order $k$ (see Gupta and Kumar (2021)). For $0\le \rho<1$ and $\lambda_{j}=\lambda(1-\rho)\rho^{j-1}/(1-\rho^{k})$,   $j=1,2,\dots,k$ it reduces to the space fractional version of the PAPoK.

Let $l_{\beta}(t,\cdot)$ be the density of stable subordinator and $G(u,t)$ be the pgf of GCP given in \eqref{pgfmt}. From \eqref{md}, the pgf of GSFCP can be obtained as
\begin{align}\label{pgfa}
G_{\beta}(u,t)&=\int_{0}^{\infty}G(u,x)l_{\beta}(t,x)\mathrm{d}x\nonumber\\
&=\int_{0}^{\infty}\exp\Big(-\sum_{j=1}^{k}\lambda_{j}(1-u^{j})x\Big)l_{\beta}(t,x)\mathrm{d}x\nonumber\\
&=\exp\bigg(-t\Big(\sum_{j=1}^{k}\lambda_{j}(1-u^{j})\Big)^{\beta}\bigg).
\end{align}
\begin{remark}\label{fs}
	Alternatively, the pgf \eqref{pgfa} can be obtained by choosing the Bern\v stein funstion $f(s)=f_{2}(s)=s^{\beta}$ in Proposition 6 of Kataria and Khandakar (2022a).
	
	On substituting $\lambda_{j}=\lambda$ for all $j=1,2,\dots,k$ in \eqref{pgfa}, we get the pgf of space fractional Poisson process of order $k$ (see Gupta and Kumar (2021), Eq. (34)).
\end{remark}
The pgf of GSFCP satisfies the following differential equation:
\begin{equation*}
\frac{\partial}{\partial t}G_{\beta}(u,t)=-\bigg(\sum_{j=1}^{k}\lambda_{j}(1-u^{j})\bigg)^{\beta}G_{\beta}(u,t), \ G_{\beta}(u,0)=1.
\end{equation*}
\begin{proposition}
	Let $X_{i}$, $i\ge1$ be iid uniform random variables in $[0,1]$. Then,
	\begin{equation*}
		G_{\beta}(u,t)=\mathrm{Pr}\bigg\{\min_{0\le i \le N(t)} X_{i}^{1/\beta}\ge 1-\frac{1}{\Lambda}\sum_{j=1}^{k}\lambda_{j}u^{j}\bigg\},\ 0<u<1,
	\end{equation*}
	where $\{N(t)\}_{t\ge0}$ is a homogeneous Poisson process with rate $\Lambda^{\beta}$ such that $\min_{0\le i \le N(t)} X_{i}^{1/\beta}= 1$ when $N(t)=0$.
\end{proposition}
\begin{proof}
	Observe that
	\begin{align*}
		\mathrm{Pr}\bigg\{\min_{0\le i \le N(t)} X_{i}^{1/\beta}\ge 1-\frac{1}{\Lambda}\sum_{j=1}^{k}\lambda_{j}u^{j}\bigg\}&=\sum_{n=0}^{\infty}\mathrm{Pr}\bigg\{\min_{0\le i \le n} X_{i}^{1/\beta}\ge 1-\frac{1}{\Lambda}\sum_{j=1}^{k}\lambda_{j}u^{j}\bigg\}\frac{(\Lambda^{\beta}t)^{n}}{n!}e^{-t\Lambda^{\beta}}\\
		&=e^{-t\Lambda^{\beta}}\sum_{n=0}^{\infty}\bigg(\mathrm{Pr}\Big\{X_{1}^{1/\beta}\ge 1-\frac{1}{\Lambda}\sum_{j=1}^{k}\lambda_{j}u^{j}\Big\}\bigg)^{n}\frac{(\Lambda^{\beta}t)^{n}}{n!}\\
		&=e^{-t\Lambda^{\beta}}\sum_{n=0}^{\infty}\bigg(1- \Big(1-\frac{1}{\Lambda}\sum_{j=1}^{k}\lambda_{j}u^{j}\Big)^{\beta}\bigg)^{n}\frac{(\Lambda^{\beta}t)^{n}}{n!}\\
		&=\exp\bigg(-t\Big(\sum_{j=1}^{k}\lambda_{j}(1-u^{j})\Big)^{\beta}\bigg).
	\end{align*}
	This completes the proof.	
\end{proof}

From (\ref{bgd43}) and Remark \ref{fs}, its jump's distribution is given by
\begin{equation*}
\mathrm{Pr}\{Y(h)=n\}=\begin{cases*}
1-h\Lambda^{\beta}+o(h), \ n=0,\\
\displaystyle- h\sum_{ \Omega(k,n)}(\beta)_{z_{k}}\Lambda^{\beta-z_{k}}\prod_{j=1}^{k}\frac{(-\lambda_{j})^{x_{j}}}{x_{j}!}+o(h), \ n\ge1,
\end{cases*}
\end{equation*}
where $(\beta)_{z_{k}}=\beta(\beta-1)\cdots(\beta-z_k+1)$.

 The L\'evy measure and the distribution of first passage times of GSFCP follows from \eqref{levymeasure} and \eqref{tm}, respectively.
\begin{remark}
	In view of \eqref{compound}, it can be observed that the GSFCP is equal in distribution to a space fractional compound Poisson process, that is,
	\begin{equation*}
	Y(t)\stackrel{d}{=}\sum_{j=1}^{N(D_{\beta}(t))}X_{j}
	\end{equation*}
	where $\{N(D_{\beta}(t))\}_{t\ge0}$ is the space fractional Poisson process with intensity $\Lambda$, and it is independent of the sequence of iid random variables $\{X_{j}\}_{j\ge 1}$ with distribution \eqref{xj}.
\end{remark}

The state probabilities of GSFCP can be obtained by putting $f(s)=s^{\beta}$ in Remark 9 of Kataria and Khandakar (2022a) as follows:
\begin{align}\label{pntg}
\mathrm{Pr}\{Y(t)=n\}&=\sum_{ \Omega(k,n)}\prod_{j=1}^{k}\frac{\lambda_{j}^{x_{j}}}{x_{j}!}\frac{(-1)^{z_{k}}}{\Lambda^{z_{k}}}\frac{\mathrm{d}^{z_{k}}}{\mathrm{d}v^{z_{k}}}e^{-t\Lambda^{\beta} v^{\beta}}\Big|_{v=1}\nonumber\\
&=\sum_{ \Omega(k,n)}\prod_{j=1}^{k}\frac{\lambda_{j}^{x_{j}}}{x_{j}!}\frac{(-1)^{z_{k}}}{\Lambda^{z_{k}}}\sum_{r=0}^{\infty}\frac{(-\Lambda^{\beta}t)^{r}}{r!}\beta r(\beta r-1)\cdots (\beta r-z_{k}+1)\nonumber\\
&=\sum_{ \Omega(k,n)}\prod_{j=1}^{k}\frac{\lambda_{j}^{x_{j}}}{x_{j}!}\frac{(-1)^{z_{k}}}{\Lambda^{z_{k}}}\sum_{r=0}^{\infty}\frac{(-\Lambda^{\beta}t)^{r}}{r!}\frac{\Gamma(\beta r+1)}{\Gamma(\beta r+1-z_{k})}\nonumber\\
&=\sum_{ \Omega(k,n)}\prod_{j=1}^{k}\frac{\lambda_{j}^{x_{j}}}{x_{j}!}\frac{(-1)^{z_{k}}}{\Lambda^{z_{k}}} \: {}_1\psi_1 \left[ \left.
\begin{array}{l}
(1,\beta) \\
(1-z_{k},\beta)
\end{array}
\right| -\Lambda^\beta t \right],\ n\ge0,
\end{align}
where ${}_1\psi_1(\cdot)$ is the generalized Wright function (see Kilbas {\it et al.} (2006),
Eq. (1.11.14)).

Further, for $k=1$ and $\Lambda=\lambda$ in \eqref{pntg}, we get the pmf of space fractional Poisson process (see Orsingher and Polito (2012), Eq. (2.15)). 

\section{GCP time-changed by L\'evy subordinator}\label{section5}
In this section, we obtain some additional results for a time-changed version of the GCP, namely, the time-changed  generalized counting process-I (TCGCP-I) $\{M_{f}(t)\}_{t\ge0}$. It is defined as follows (see Kataria and Khandakar (2022a)):
\begin{equation*}
M_{f}(t)\coloneqq M(D_{f}(t)), \ t\ge0,
\end{equation*}
where the GCP $\{M(t)\}_{t\ge0}$ is independent of the L\'evy subordinator $\{D_{f}(t)\}_{t\ge0}$. Here, the condition that the L\'evy subordinator has finite moments is not essential.

The next result follows from \eqref{bgd43}.
\begin{proposition}
	The state probabilities $p_{f}(n,t)=\mathrm{Pr}\{M_{f}(t)=n\}$, $n\ge0$ of TCGCP-I solve the following system of differential equations:
	\begin{equation}\label{pnt}
	\frac{\mathrm{d}}{\mathrm{d}t}p_{f}(n,t)=-f(\Lambda)p_{f}(n,t)-\sum_{m=1}^{n}\sum_{ \Omega(k,m)}f^{(z_{k})}(\Lambda)\prod_{j=1}^{k}\frac{(-\lambda_{j})^{x_{j}}}{x_{j}!}p_{f}(n-m,t),
	\end{equation}
	with initial conditions 
	\begin{equation*}
	p_{f}(n,0)=\begin{cases*}
	1,\  n=0,\\
	0,\  n\ge 1.
	\end{cases*}
	\end{equation*}
\end{proposition}
\begin{remark}
	On substituting $k=1$ in \eqref{pnt}, we get the system of differential equations that governs the state probabilities of a time-changed Poisson process (see Orsingher and Toaldo (2015), Eq. (2.1)).
\end{remark}

The system of differential equations given in \eqref{pnt} can equivalently written as follows:
	\begin{equation}\label{space}
	\frac{\mathrm{d}}{\mathrm{d}t}p_{f}(n,t)=-f\bigg(\Lambda\Big(I-\frac{1}{\Lambda}\sum_{j=1}^{n\wedge k}\lambda_{j}B^{j}\Big)\bigg)p_{f}(n,t),
	\end{equation}
	where $B$ is backward shift operator such that $Bp_{f}(n,t)=p_{f}(n-1,t)$. This can be shown as follows:
\scriptsize	\begin{align}\label{123}
	-f\bigg(\Lambda&\Big(I-\frac{1}{\Lambda}\sum_{j=1}^{n\wedge k}\lambda_{j}B^{j}\Big)\bigg)p_{f}(n,t)\nonumber\\
	&=-\int_{0}^{\infty}\left(I-e^{-\Lambda s\left(I-\frac{1}{\Lambda}\sum_{j=1}^{ n\wedge k}\lambda_{j}B^{j}\right)}\right)\overline{\nu}(\mathrm{d}s)p_{f}(n,t)\nonumber\\
	&=-\int_{0}^{\infty}\bigg(p_{f}(n,t)-e^{-\Lambda s}\sum_{r=0}^{\infty}\frac{s^{r}}{r!}\Big(\sum_{j=1}^{ n\wedge k}\lambda_{j}B^{j}\Big)^{r}p_{f}(n,t)\bigg)\overline{\nu}(\mathrm{d}s)\nonumber\\
	&=-\int_{0}^{\infty}\bigg(p_{f}(n,t)-e^{-\Lambda s}\sum_{r=0}^{\infty}s^{r}\underset{x_{1}+x_{2}+\dots+x_{n\wedge k}=r}{\sum}\prod_{j=1}^{n\wedge k}\frac{\lambda_{j}^{x_{j}}}{x_{j}!}B^{jx_{j}}p_{f}(n,t)\bigg)\overline{\nu}(\mathrm{d}s)\nonumber\\
	&=-\int_{0}^{\infty}\bigg(p_{f}(n,t)-e^{-\Lambda s}\sum_{r=0}^{\infty}s^{r}\underset{x_{1}+x_{2}+\dots+x_{k}=r}{\sum}\prod_{j=1}^{ k}\frac{\lambda_{j}^{x_{j}}}{x_{j}!}B^{jx_{j}}p_{f}(n,t)\bigg)\overline{\nu}(\mathrm{d}s),\ (\text{as} \ p_{f}(-j,t)=0, \ j\ge 1)\nonumber\\
	&=-\int_{0}^{\infty}\bigg(p_{f}(n,t)-e^{-\Lambda s}\sum_{r=0}^{\infty}s^{r}\sum_{m=r}^{\infty}\underset{x_{1}+2x_{2}+\dots+kx_{k}=m}{\sum_{ x_{1}+x_{2}+\dots+x_{k}=r}}\prod_{j=1}^{k}\frac{\lambda_{j}^{x_{j}}}{x_{j}!}B^{jx_{j}}p_{f}(n,t)\bigg)\overline{\nu}(\mathrm{d}s)\nonumber\\
	&=-\int_{0}^{\infty}\bigg(p_{f}(n,t)-e^{-\Lambda s}\sum_{m=0}^{\infty}\sum_{r=0}^{m}s^{r}\underset{x_{1}+2x_{2}+\dots+kx_{k}=m}{\sum_{ x_{1}+x_{2}+\dots+x_{k}=r}}\prod_{j=1}^{k}\frac{\lambda_{j}^{x_{j}}}{x_{j}!}B^{jx_{j}}p_{f}(n,t)\bigg)\overline{\nu}(\mathrm{d}s)\nonumber\\
	&=-\int_{0}^{\infty}\bigg(p_{f}(n,t)-e^{-\Lambda s}\sum_{m=0}^{\infty}\sum_{ \Omega(k,m)}s^{z_{k}}\prod_{j=1}^{k}\frac{\lambda_{j}^{x_{j}}}{x_{j}!}B^{m}p_{f}(n,t)\bigg)\overline{\nu}(\mathrm{d}s)\nonumber\\
	&=-\int_{0}^{\infty}\bigg(p_{f}(n,t)-\sum_{m=0}^{\infty}\sum_{ \Omega(k,m)}\prod_{j=1}^{k}\frac{\lambda_{j}^{x_{j}}}{x_{j}!}p_{f}(n-m,t)s^{z_{k}}e^{-\Lambda s}\bigg)\overline{\nu}(\mathrm{d}s)\nonumber\\
	&=-\int_{0}^{\infty}\left(1-e^{-\Lambda s}\right)p_{f}(n,t)\overline{\nu}(\mathrm{d}s)+\sum_{m=1}^{n}\sum_{ \Omega(k,m)}\prod_{j=1}^{k}\frac{\lambda_{j}^{x_{j}}}{x_{j}!}p_{f}(n-m,t)\int_{0}^{\infty}s^{z_{k}}e^{-\Lambda s}\overline{\nu}(\mathrm{d}s)\nonumber\\
	&=-f(\Lambda)p_{f}(n,t)-\sum_{m=1}^{n}\sum_{ \Omega(k,m)}f^{(z_{k})}(\Lambda)\prod_{j=1}^{k}\frac{(-\lambda_{j})^{x_{j}}}{x_{j}!}p_{f}(n-m,t).
	\end{align}
\normalsize
	\begin{remark}
	If we choose $f(s)=f_{2}(s)=s^{\beta}$ in \eqref{space}, we get the system of differential equations that governs the state probabilities of GSFCP in the following form:
	\begin{equation}\label{spacep}
	\frac{\mathrm{d}}{\mathrm{d}t}p_{f_{2}}(n,t)=-\Lambda^{\beta}\bigg(I-\frac{1}{\Lambda}\sum_{j=1}^{n\wedge k}\lambda_{j}B^{j}\bigg)^{\beta}p_{f_{2}}(n,t),
	\end{equation}
	with $p_{f_{2}}(n,0)=\delta_{n,0}$. Also, on substituting $\lambda_{j}=\lambda$ for all $1\le j\le k$ in \eqref{spacep}, we get the system of differential equations that governs the state probabilities of the space fractional Poisson process of order $k$.
\end{remark}

 For the Bern\v stein function $f_1$ given in \eqref{gamma}, the TCGCP-1 reduces to the GCP time-changed by an independent gamma subordinator, that is,
\begin{equation}\label{mg}
	M_{f_{1}}(t)\coloneqq M(Z(t)), \ t\ge0.
\end{equation}
\begin{proposition}
Let $e^{-\partial _{t}/a}$ be the shift operator defined in \eqref{shi}. Then, the pmf $p_{f_{1}}(n,t)=\mathrm{Pr}\{M_{f_{1}}(t)=n\}$, $n\ge0$ of $\{M_{f_{1}}(t)\}_{t\ge0}$ satisfies the following system:
	\begin{equation*}
		e^{-\partial _{t}/a}p_{f_{1}}(n,t)=	(1+\Lambda/b)p_{f_{1}}(n,t)-\frac{1}{b}\sum_{j=1}^{n\wedge k}\lambda_{j}p_{f_{1}}(n-j,t),
	\end{equation*}
	with initial conditions
	\begin{equation*}
		p_{f_{1}}(n,0)=\begin{cases*}
			1, \ n=0,\\
			0, \ n>1.
		\end{cases*}
	\end{equation*}
\end{proposition}
\begin{proof}
	From \eqref{mg}, we can write
	\begin{equation}\label{jhi}
		p_{f_{1}}(n,t)=\int_{0}^{\infty}p(n,x)g(x,t)\mathrm{d}z,
	\end{equation}
where $p(n,x)$ is the pmf of GCP and $g(x,t)$ is the density function of gamma subordinator.
	By applying the shift operator $e^{-\partial _{t}/a}$ on both sides of \eqref{jhi} and using \eqref{res}, we get
	\begin{align*}
		e^{-\partial _{t}/a}p_{f_{1}}(n,t)&=\int_{0}^{\infty}p(n,x)\left(g(x,t)+\frac{1}{b}\frac{\partial }{\partial x}g(x,t)\right)\mathrm{d}x\\
		&=p_{f_{1}}(n,t)+\frac{1}{b}\left[p(n,x)g(x,t)\right]_{x=0}^{\infty}-\frac{1}{b}\int_{0}^{\infty}g(x,t)\frac{\mathrm{d}}{\mathrm{d}x}p(n,x)\mathrm{d}x\\
		&=p_{f_{1}}(n,t)-\frac{1}{b}\int_{0}^{\infty}g(x,t)\bigg(-\Lambda p(n,x)+	\sum_{j=1}^{n\wedge k}\lambda_{j}p(n-j,x)\bigg)\mathrm{d}x,
	\end{align*}
where in the last step we have used \eqref{cre} with $\beta=1$.	The proof concludes in view of \eqref{jhi}.
\end{proof}

\subsection{First passage times of TCGCP-I}
Here, we study the first passage times of TCGCP-I. It is defined as
\begin{equation*}
	T_{f}^{n}\coloneqq\inf\{s\ge0:M_{f}(s)\ge n\}.
\end{equation*}
The distribution of $T_{f}^{n}$ can be written as
\begin{equation*}
	\mathrm{Pr}\{T_{f}^{n}<s\}=\mathrm{Pr}\{M_{f}(s)\ge n\}=\sum_{m=n}^{\infty}\int_{0}^{\infty}\sum_{\Omega(k,m)}\prod_{j=1}^{k}\frac{(\lambda_{j}x)^{x_{j}}}{x_{j}!}e^{-\Lambda x}\mathrm{Pr}\{D_{f}(s)\in \mathrm{d}x\}.
\end{equation*}
So, its density function $\mathscr{P}_{f}(n,s)=\mathrm{Pr}\{T_{f}^{n}\in \mathrm{d}s\}/\mathrm{d}s$ can be obtained as
\begin{align}
\mathscr{P}_{f}(n,s)&=\frac{\mathrm{d}}{\mathrm{d}s}\sum_{m=n}^{\infty}\int_{0}^{\infty}\sum_{\Omega(k,m)}\prod_{j=1}^{k}\frac{(\lambda_{j}x)^{x_{j}}}{x_{j}!}e^{-\Lambda x}\mathrm{Pr}\{D_{f}(s)\in \mathrm{d}x\}\label{fgt}\\
	&=\frac{\mathrm{d}}{\mathrm{d}s}\int_{0}^{\infty}\left(1-\mathrm{Pr}\{M(x)\le n-1\}\right)\mathrm{Pr}\{D_{f}(s)\in \mathrm{d}x\}\nonumber\\
	&=\frac{\mathrm{d}}{\mathrm{d}s}\int_{0}^{\infty}\bigg(1-\sum_{l=0}^{n-1}\sum_{\Omega(k,l)}\prod_{j=1}^{k}\frac{(\lambda_{j}x)^{x_{j}}}{x_{j}!}e^{-\Lambda x}\bigg)\mathrm{Pr}\{D_{f}(s)\in \mathrm{d}x\}\nonumber\\
	&=-\frac{\mathrm{d}}{\mathrm{d}s}\sum_{l=0}^{n-1}\sum_{\Omega(k,l)}\prod_{j=1}^{k}\frac{(-\lambda_{j})^{x_{j}}}{x_{j}!}\int_{0}^{\infty}\frac{\mathrm{d}^{z_{k}}}{\mathrm{d}\Lambda^{z_{k}}}e^{-\Lambda x}\mathrm{Pr}\{D_{f}(s)\in \mathrm{d}x\}\nonumber\\
	&=-\frac{\mathrm{d}}{\mathrm{d}s}\sum_{l=0}^{n-1}\sum_{\Omega(k,l)}\prod_{j=1}^{k}\frac{(-\lambda_{j})^{x_{j}}}{x_{j}!}\frac{\mathrm{d}^{z_{k}}}{\mathrm{d}\Lambda^{z_{k}}}e^{-sf(\Lambda)}\label{hitting}.
\end{align}
Equivalently, 
\begin{equation*}
	\mathrm{Pr}\{T_{f}^{n}\in \mathrm{d}s\}=\mathrm{Pr}\{T_{f}^{n-1}\in \mathrm{d}s\}-\sum_{\Omega(k,n-1)}\prod_{j=1}^{k}\frac{(-\lambda_{j})^{x_{j}}}{x_{j}!}\frac{\mathrm{d}}{\mathrm{d}s}\frac{\mathrm{d}^{z_{k}}}{\mathrm{d}\Lambda^{z_{k}}}e^{-sf(\Lambda)}\mathrm{d}s,
\end{equation*}
which gives a recursive relation between the distributions of $T_{f}^{n}$.
\begin{remark}
	From \eqref{hitting}, we have
	\begin{equation*}
	\mathrm{Pr}\{T_{f}^{1}\in \mathrm{d}s\}=f(\Lambda)e^{-sf(\Lambda)}\mathrm{d}s,	
	\end{equation*}
that is, the waiting time of the first event for TCGCP-I is exponential. 
\end{remark}
Next, we derive the system of differential equations that governs the distribution of $T_{f}^{n}$. The following result will be used:
\begin{proposition}\label{prop}
	Let $\mathcal{O}_{f}(u,t)=\sum_{n=1}^{\infty}u^{n}\mathscr{P}_{f}(n,t)$, $|u|<1$. Then,
	\begin{equation*}
		\frac{\mathrm{d}}{\mathrm{d}t}\mathcal{O}_{f}(u,t)=-f\Big(\sum_{j=1}^{k}(1-u^{j})\lambda_{j}\Big)\mathcal{O}_{f}(u,t).
	\end{equation*}
\end{proposition}
\begin{proof}
	On using \eqref{fgt}, we get
\begin{align*}
	\mathcal{O}_{f}(u,t)&=\frac{\mathrm{d}}{\mathrm{d}t}\sum_{m=1}^{\infty}\sum_{n=1}^{m}u^{n}\int_{0}^{\infty}\sum_{\Omega(k,m)}\prod_{j=1}^{k}\frac{(\lambda_{j}x)^{x_{j}}}{x_{j}!}e^{-\Lambda x}\mathrm{Pr}\{D_{f}(t)\in \mathrm{d}x\}\\
	&=\frac{\mathrm{d}}{\mathrm{d}t}\frac{u}{u-1}\sum_{m=1}^{\infty}(u^{m}-1)\int_{0}^{\infty}\sum_{\Omega(k,m)}\prod_{j=1}^{k}\frac{(\lambda_{j}x)^{x_{j}}}{x_{j}!}e^{-\Lambda x}\mathrm{Pr}\{D_{f}(t)\in \mathrm{d}x\}\\
	&=\frac{\mathrm{d}}{\mathrm{d}t}\frac{u}{u-1}\int_{0}^{\infty}\Big(e^{\sum_{j=1}^{k}(u^{j}-1)\lambda_{j}x}-1\Big)\mathrm{Pr}\{D_{f}(t)\in \mathrm{d}x\}\\
	&=\frac{\mathrm{d}}{\mathrm{d}t}\frac{u}{u-1}e^{-tf\big(\sum_{j=1}^{k}(1-u^{j})\lambda_{j}\big)}\\
	&=\frac{u}{1-u}f\Big(\sum_{j=1}^{k}(1-u^{j})\lambda_{j}\Big)e^{-tf\big(\sum_{j=1}^{k}(1-u^{j})\lambda_{j}\big)}.
\end{align*}
The result follows on taking the derivative with respect to $t$.
\end{proof}
\begin{theorem}
	The density function of the first passage times of TCGCP-I solves the following system of differential equations:
	\begin{equation*}
\frac{\mathrm{d}}{\mathrm{d}t}\mathscr{P}_{f}(n,t)=-f\bigg(\Lambda\Big(I-\frac{1}{\Lambda}\sum_{j=1}^{n\wedge k}\lambda_{j}B^{j}\Big)\bigg)\mathscr{P}_{f}(n,t),\ n\ge1.		
	\end{equation*}
\end{theorem}
\begin{proof}
On using \eqref{123} for $\mathscr{P}_{f}(n,t)$, we get
\begin{align*}
\sum_{n=1}^{\infty}u^{n}f&\bigg(\Lambda\Big(I-\frac{1}{\Lambda}\sum_{j=1}^{n\wedge k}\lambda_{j}B^{j}\Big)\bigg)\mathscr{P}_{f}(n,t)\\
&=f(\Lambda)\mathcal{O}_{f}(u,t)-\sum_{n=1}^{\infty}u^{n}\sum_{m=1}^{n}\sum_{ \Omega(k,m)}\prod_{j=1}^{k}\frac{\lambda_{j}^{x_{j}}}{x_{j}!}\mathscr{P}_{f}(n-m,t)\int_{0}^{\infty}s^{z_{k}}e^{-\Lambda s}\overline{\nu}(\mathrm{d}s)\\
&=f(\Lambda)\mathcal{O}_{f}(u,t)-\sum_{m=1}^{\infty}\sum_{n=m}^{\infty}u^{n}\sum_{ \Omega(k,m)}\prod_{j=1}^{k}\frac{\lambda_{j}^{x_{j}}}{x_{j}!}\mathscr{P}_{f}(n-m,t)\int_{0}^{\infty}s^{z_{k}}e^{-\Lambda s}\overline{\nu}(\mathrm{d}s)\\
&=f(\Lambda)\mathcal{O}_{f}(u,t)-\mathcal{O}_{f}(u,t)\sum_{m=1}^{\infty}u^{m}\sum_{ \Omega(k,m)}\prod_{j=1}^{k}\frac{\lambda_{j}^{x_{j}}}{x_{j}!}\int_{0}^{\infty}s^{z_{k}}e^{-\Lambda s}\overline{\nu}(\mathrm{d}s)\\
&=f(\Lambda)\mathcal{O}_{f}(u,t)-\mathcal{O}_{f}(u,t)\int_{0}^{\infty}e^{-\Lambda s}\Big(e^{\sum_{j=1}^{k}u^{j}\lambda_{j}s}-1\Big)\overline{\nu}(\mathrm{d}s)\\
&=f\Big(\sum_{j=1}^{k}(1-u^{j})\lambda_{j}\Big)\mathcal{O}_{f}(u,t).
\end{align*}
The result follows on using Proposition \ref{prop}.
\end{proof}

	Let $\mathcal{T}_{f}^{n}$ be the first-hitting times of TCGCP-I, that is,
\begin{equation*}
	\mathcal{T}_{f}^{n}\coloneqq \inf \left\{s\ge0:M_{f}(s)=n\right\},\ n\ge1.
\end{equation*}
As the TCGCP-I is a L\'evy process, it has independent increments. So, the hitting time distribution for state $n$ of TCGCP-I has the following form:
\begin{align*}
\mathrm{Pr}\left\{\mathcal{T}_{f}^{n}\in \mathrm{d}s\right\}&=\mathrm{Pr}\bigg\{\bigcup_{l=0}^{n-1}\{M_{f}(s)=l, M_{f}[s, s+\mathrm{d}s)=n-l)\}\bigg\}\\
&=\sum_{l=0}^{n-1}\sum_{\Omega(k,l)}\prod_{j=1}^{k}\frac{\lambda_{j}^{x_{j}}}{x_{j}!}\frac{(-1)^{z_{k}}}{\Lambda^{z_{k}}}\frac{\mathrm{d}^{z_{k}}}{\mathrm{d}v^{z_{k}}}e^{-sf(\Lambda v)}\Big|_{v=1} \cdot(-\mathrm{d}s)\sum_{ \Omega(k,n-l)}f^{(z_{k})}(\Lambda)\prod_{j=1}^{k}\frac{(-\lambda_{j})^{x_{j}}}{x_{j}!},
\end{align*}
where we have used Remark 9 of Kataria and Khandakar (2022a) and \eqref{bgd43} in the last step.
Therefore, 
\begin{equation}\label{tmf}
\mathrm{Pr}\left\{\mathcal{T}_{f}^{n}<\infty\right\}=\sum_{l=0}^{n-1}\sum_{\Omega(k,l)}\prod_{j=1}^{k}\frac{\lambda_{j}^{x_{j}}}{x_{j}!}\frac{(-1)^{z_{k}+1}}{\Lambda^{z_{k}}}\frac{\mathrm{d}^{z_{k}}}{\mathrm{d}v^{z_{k}}}\frac{1}{f(\Lambda v)}\bigg|_{v=1}\sum_{ \Omega(k,n-l)}f^{(z_{k})}(\Lambda)\prod_{j=1}^{k}\frac{(-\lambda_{j})^{x_{j}}}{x_{j}!}.
\end{equation}
On taking $k=1$ and $\Lambda=\lambda$ in \eqref{tmf}, we get
	\begin{equation*}
	\mathrm{Pr}\left\{\mathcal{T}_{f}^{n}<\infty\right\}\big|_{k=1}=\sum_{l=0}^{n-1}\frac{(-1)^{n+1}}{l!(n-l)!}\frac{\mathrm{d}^{l}}{\mathrm{d}v^{l}}\frac{1}{f(\lambda v)}\Big|_{v=1} \lambda^{n-l}f^{(n-l)}(\lambda),
	\end{equation*}
	which agrees with Eq. (2.7) of Garra {\it et al.} (2017).

\subsection{An application to risk theory}
Consider the following risk model with TCGCP-I as the counting process:
\begin{equation}\label{riskmodel}
	X(t)= ct-\sum_{j=1}^{M_{f}(t)}Z_j, \ t\geq 0,
\end{equation}
where $c>0$ denotes the constant premium rate and $ \{Z_j\}_{j\ge1}$ is a sequence of positive iid random variables with distribution $F$. Here, $Z_j$'s represent the claim sizes which are independent of $\{M_{f}(t)\}_{t\ge0}$. The expected value of $\{X(t)\}_{t\ge0}$ is given by
\begin{equation*}
	\mathbb{E}(X(t))=ct-\mu \mathbb{E}(M_f(t)),
\end{equation*}
where $\mu=\mathbb{E}(Z_j)$. The relative safety loading factor $\rho$ for \eqref{riskmodel} is given by
\begin{equation*}
	\rho = \frac{ \mathbb{E}(X(t))}{\mathbb{E}\big(\sum_{j=1}^{M_f(t)} Z_j\big)} = \frac{ct}{\mu \mathbb{E}(M_f(t))}-1.
\end{equation*} 

 Let $U(t)= u+X(t)$, $t\ge0$ be the surplus process 
where $u\ge0$ is the initial capital, $\tau$ be the time to ruin of an insurance company, that is, $ \tau=\inf\{t>0: U(t)<0\}$ with $\inf \phi=\infty$ and $\psi(u)=\mathrm{Pr}\{\tau<\infty\}$ be the ruin probability. Let 
\begin{equation}\label{jointpdf}
	K(u,y)= \mathrm{Pr}\{\tau<\infty, D\leq y \},\ y\ge0,
\end{equation}
be the joint probability that the ruin occurs in finite time and the deficit, that is, $D=|U(\tau)|$
at the time of ruin is not more than $y$. Also, let $u'=u+ch$ and $F^{*j}$ be the distribution of $Z_{1}+Z_{2}+\dots+Z_{j}$ for all $j\ge 1$. Now, using \eqref{bgd43} and \eqref{jointpdf}, we get
\begin{align*}
	K(u,y)=& (1-hf(\Lambda))K(u',y)+o(h) -h\sum_{n=1}^{\infty} \sum_{\Omega(k,n)}f^{(z_{k})}(\Lambda)\prod_{j=1}^{k}\frac{(-\lambda_{j})^{x_{j}}}{x_{j}!}\\
	&\hspace*{1cm}\times \bigg(\int_0^{u'}K(u'-x,y)\mathrm{d}F^{*n}(x) + F^{*n}(u'+y)-F^{*n}(u')\bigg).
\end{align*}
After rearranging the terms, we have that
\begin{align*}
	\frac{K(u',y)-K(u,y)}{ch}
	=& \frac{f(\Lambda)}{c} K(u',y)+\frac{o(h)}{h}-\frac{1}{c} \sum_{n=1}^{\infty} \sum_{\Omega(k,n)}f^{(z_{k})}(\Lambda)\prod_{j=1}^{k}\frac{(-\lambda_{j})^{x_{j}}}{x_{j}!}\\
	&\hspace*{1cm}\times  \bigg(\int_0^{u'}K(u'-x,y)\mathrm{d}F^{*n}(x) + F^{*n}(u'+y)-F^{*n}(u')\bigg).
\end{align*}
On taking $ h\rightarrow 0$, we get
\begin{align}\label{hgt}
	\frac{\partial K(u,y)}{\partial u}& = \frac{f(\Lambda)}{c}K(u,y)-\frac{1}{c} \sum_{n=1}^{\infty} \sum_{\Omega(k,n)}f^{(z_{k})}(\Lambda)\prod_{j=1}^{k}\frac{(-\lambda_{j})^{x_{j}}}{x_{j}!}\nonumber\\
	&\hspace*{1cm}\times\bigg(\int_0^{u}K(u-x,y)\mathrm{d}F^{*n}(x) + F^{*n}(u+y)-F^{*n}(u)\bigg).
\end{align}
Note that
\begin{equation*}
	\sum_{n=1}^{\infty} \sum_{\Omega(k,n)}f^{(z_{k})}(\Lambda)\prod_{j=1}^{k}\frac{(-\lambda_{j})^{x_{j}}}{x_{j}!}= \sum_{r=0}^{\infty}\frac{(-\Lambda)^r}{r!} f^{(r)}(\Lambda) -f(\Lambda)\\
	=f(0)-f(\Lambda)\\
	=-f(\Lambda),
\end{equation*}
 as $f(0)=0$. Now, let 
\begin{equation*}
	W(x)=-\frac{1}{f(\Lambda)}\sum_{n=1}^{\infty} \sum_{\Omega(k,n)}f^{(z_{k})}(\Lambda)\prod_{j=1}^{k}\frac{(-\lambda_{j})^{x_{j}}}{x_{j}!}F^{*n}(x)
\end{equation*}
be the mixer distribution of the aggregated claims. Thus, \eqref{hgt} reduces to
	\begin{equation}\label{appl:thmde}
		\frac{\partial K(u,y)}{\partial u} = \frac{f(\Lambda)}{c}\left( K(u,y)+W(u)-W(u+y)-\int_0^{u}K(u-x,y)\mathrm{d}W(x) \right).
	\end{equation}
 As $\lim_{y\rightarrow \infty} K(u,y)=\psi(u)$, so by letting $y\to \infty$ in \eqref{appl:thmde}, we get the following governing equation of the ruin probability:
	\begin{equation*}
	\frac{\mathrm{d}}{\mathrm{d} u}\psi(u)= \frac{f(\Lambda)}{c}\left(\psi(u)+W(u)-1-\int_0^{u}\psi(u-x)\mathrm{d}W(x) \right).
\end{equation*}
\begin{theorem}
The joint probability of ruin time and deficit at the time of ruin with zero initial capital is given by
	\begin{equation}\label{appl:thm2}
		K(0,y)= \frac{f(\Lambda)}{c} \int_0^{y}\left(1-W(u)\right)\mathrm{d}u.
	\end{equation}
\end{theorem}
\begin{proof}
	On integrating \eqref{appl:thmde} with respect to $u$ on $ (0,\infty)$, we get
	\begin{align*}
		K(\infty,y)-K(0,y)&= \frac{f(\Lambda)}{c} \left(\int_0^{\infty}K(u,y)\mathrm{d}u+\int_0^{\infty}\left(W(u)-W(u+y)\right)\mathrm{d}u\right)\\
		&\ \ - \frac{f(\Lambda)}{c}\int_0^{\infty} \int_0^{u}K(u-x,y)\mathrm{d}W(x)\mathrm{d}u .
	\end{align*}
Thus, using $K(\infty,y)=0$ and the change of variable yields	the following
\begin{equation*}
	K(0,y)=\frac{f(\Lambda)}{c} \int_0^{\infty}\left(W(u+y)-W(u)\right)\mathrm{d}u=\frac{f(\Lambda)}{c} \int_0^{y}\left(1-W(u)\right)\mathrm{d}u.
\end{equation*}
This completes the proof.
\end{proof}
\begin{remark}
On taking $y\to\infty $ in \eqref{appl:thm2}, we get	$\psi(0) = \frac{f(\Lambda)}{c}\int_0^{\infty}(1-W(u))\mathrm{d}u$.
\end{remark}

\subsection{GFCP time-changed by L\'evy subordinator}
Kataria and Khandakar (2022a) introduced the following  time-changed process, namely, the TCGFCP-I: 
\begin{equation}\label{zft}
	M_{f}^{\beta}(t)\coloneqq M^{\beta}(D_{f}(t)), \ t\ge0,
\end{equation}
where the GFCP $\{M^{\beta}(t)\}_{t\ge0}$ is independent of the L\'evy subordinator $\{D_{f}(t)\}_{t\ge0}$ whose moments are finite.

 First, we give an additional result for $\{M_{f}^{\beta}(t)\}_{t\ge0}$.
 \begin{proposition}
	The one-dimensional distributions of TCGFCP-I are not infinitely divisible.
\end{proposition}
\begin{proof}
	Using the self-similarity property of inverse stable subordinator, we get
	\begin{equation*}
	M_{f}^{\beta}(t)=M(Y_{\beta}(D_{f}(t)))\stackrel{d}{=}M\left((D_{f}(t))^{\beta}Y_{\beta}(1)\right).
	\end{equation*}
As  $D_{f}(t)\to \infty$ as $t\to \infty$, almost surely,	we have
	\begin{align*}
	\lim\limits_{t\to\infty}\frac{M_{f}^{\beta}(t)}{t^{\beta}}&\stackrel{d}{=}\lim\limits_{t\to\infty}\frac{M\left((D_{f}(t))^{\beta}Y_{\beta}(1)\right)}{t^{\beta}}\\
	&=Y_{\beta}(1)\lim\limits_{t\to\infty}\frac{M\left((D_{f}(t))^{\beta}Y_{\beta}(1)\right)}{(D_{f}(t))^{\beta}Y_{\beta}(1)}\Big(\frac{D_{f}(t)}{t}\Big)^{\beta}\\
	&\stackrel{d}{=}\sum_{j=1}^{k}j\lambda_{j}Y_{\beta}(1)\lim\limits_{t\to\infty}\Big(\frac{D_{f}(t)}{t}\Big)^{\beta}, \ (\text{using}\ \eqref{limit})\\
	&\stackrel{d}{=}\sum_{j=1}^{k}j\lambda_{j}Y_{\beta}(1)\left(\mathbb{E}\left(D_{f}(1)\right)\right)^{\beta},
	\end{align*}
where in the last step we have used the strong law of large numbers for a L\'evy subordinator (see Bertoin (1996), p. 92). Thus, the result follows as $Y_{\beta}(1)$ is not infinitely divisible (see Vellaisamy and Kumar (2018a)).
\end{proof}
 In (\ref{zft}), we take the Bern\v stein function $f_1$  to obtain the GFCP time-changed by an independent gamma subordinator, that is,
\begin{equation*}
		M_{f_{1}}^{\beta}(t)\coloneqq M^{\beta}
	(Z(t)), \ t\ge0.
\end{equation*}

The mean, variance, covariance and LRD property of $\{M_{f_{1}}^{\beta}(t)\}_{t\ge0}$ follows from Section 4.1 of Kataria and Khandakar (2022a). 
It can be shown that the increment process $\{M_{f_{1}}^{\beta}(t+h)-M_{f_{1}}^{\beta}(t)\}_{t\ge0}$, $h>0$ is fixed, exhibits the short-range dependence property. The proof follows similar lines to that of Theorem 4 of Maheshwari
and Vellaisamy (2016).

For $k=1$, the process $\{M_{f_{1}}^{\beta}(t)\}_{t\ge0}$ reduces to the time fractional negative binomial process  $\{\mathcal{N}^{1}_{\beta}(t)\}_{t\ge0}$ with intensity $\lambda$ (see Vellaisamy and Maheshwari (2018b)). For notational convenience, we denote $Q^{\beta}(t,\lambda)=\mathcal{N}^{1}_{\beta}(t)$, $t\ge0$.
Using \eqref{pgftfnbp}, its $m$th factorial moment is given by
\begin{equation}\label{facn}
\mathbb{E}\left(Q^{\beta}(t,\lambda)(Q^{\beta}(t,\lambda)-1)\cdots(Q^{\beta}(t,\lambda)-m+1)\right)=\frac{\lambda^{m}m!\Gamma (bt+m\beta)}{\Gamma(m\beta+1)a^{m\beta}\Gamma(bt)},\ m\ge 1.
\end{equation}
On putting $\beta=1$ in \eqref{facn}, we get the $m$th factorial moment of negative binomial process (see Orsingher and Toaldo (2015), Remark 4.7).

Next result gives the conditional distribution of $X_{(k)}^{Q^{\beta}(t,\lambda)}$, a $k$th order statistic from the sample of size $Q^{\beta}(t,\lambda)$ where $k\in \{1,2,\dots, Q^{\beta}(t,\lambda)\}$.
\begin{theorem}\label{kthorder}
Let $\{X_{i}\}_{i\ge1}$ be a sequence of iid random variables with distribution function $F$. Then,
	\begin{equation*}
\mathrm{Pr}\{X_{(k)}^{Q^{\beta}(t,\lambda)}<z\big|Q^{\beta}(t,\lambda)\ge k\}=\frac{\mathrm{Pr}\{Q^{\beta}(t,\lambda F(z))\ge k\}}{\mathrm{Pr}\{Q^{\beta}(t,\lambda)\ge k\}}.	
	\end{equation*}
\end{theorem}
\begin{proof}
For $1\le k\le n$, we will use the following standard result for the distribution of $k$th order statistics:
	\begin{equation}\label{stan}
\mathrm{Pr}\{X_{(k)}^{n}<z\}=\sum_{j=k}^{n}\binom{n}{j}F^{j}(z)\left(1-F(z)\right)^{n-j}.
	\end{equation}
	Now,
	\begin{align}\label{orderstat}
\mathrm{Pr}\{X_{(k)}^{Q^{\beta}(t,\lambda)}<z\big|Q^{\beta}(t,\lambda)\ge k\}&=\frac{\sum_{n=k}^{\infty}\mathrm{Pr}\{X_{(k)}^{Q^{\beta}(t,\lambda)}<z,Q^{\beta}(t,\lambda)=n\}}{\mathrm{Pr}\{Q^{\beta}(t,\lambda)\ge k\}}\nonumber\\
&=\frac{\sum_{n=k}^{\infty}\mathrm{Pr}\{X_{(k)}^{Q^{\beta}(t,\lambda)}<z\big|Q^{\beta}(t,\lambda)=n\}	\mathrm{Pr}\{Q^{\beta}(t,\lambda)=n\}}{\mathrm{Pr}\{Q^{\beta}(t,\lambda)\ge k\}}.
	\end{align}
Using \eqref{tfnbppm} and \eqref{stan},	the quantity in the numerator of the right hand side of \eqref{orderstat} can be evaluated as follows:
	\begin{align*}
\sum_{n=k}^{\infty}&\mathrm{Pr}\{X_{(k)}^{Q^{\beta}(t,\lambda)}<z\big|Q^{\beta}(t,\lambda)=n\}	\mathrm{Pr}\{Q^{\beta}(t,\lambda)=n\}\\
&=\sum_{n=k}^{\infty}\bigg(\sum_{j=k}^{n}\binom{n}{j}F^{j}(z)\left(1-F(z)\right)^{n-j}\bigg)\sum_{l=n}^{\infty}(-1)^{l+n}\binom{l}{n}\frac{\Gamma(l\beta+bt)}{\Gamma(bt)\Gamma(\beta l+1)}\left(\frac{\lambda}{a^{\beta}}\right)^{l}\\
&=\sum_{l=k}^{\infty}\frac{(-1)^{l}\Gamma(l\beta+bt)}{\Gamma(bt)\Gamma(\beta l+1)}\left(\frac{\lambda}{a^{\beta}}\right)^{l}\sum_{n=k}^{l}\binom{l}{n}(-1)^{n}\sum_{j=k}^{n}\binom{n}{j}F^{j}(z)\left(1-F(z)\right)^{n-j}\\
&=\sum_{l=k}^{\infty}\frac{(-1)^{l}\Gamma(l\beta+bt)}{\Gamma(bt)\Gamma(\beta l+1)}\left(\frac{\lambda}{a^{\beta}}\right)^{l}\sum_{j=k}^{l}F^{j}(z)\sum_{n=j}^{l}\binom{l}{n}\binom{n}{j}(-1)^{n}\left(1-F(z)\right)^{n-j}\\
&=\sum_{l=k}^{\infty}\frac{(-1)^{l}\Gamma(l\beta+bt)}{\Gamma(bt)\Gamma(\beta l+1)}\left(\frac{\lambda}{a^{\beta}}\right)^{l}\sum_{j=k}^{l}\binom{l}{j}F^{j}(z)\sum_{n=j}^{l}\binom{l-j}{n-j}(-1)^{n}\left(1-F(z)\right)^{n-j}\\
&=\sum_{l=k}^{\infty}\frac{(-1)^{l}\Gamma(l\beta+bt)}{\Gamma(bt)\Gamma(\beta l+1)}\left(\frac{\lambda}{a^{\beta}}\right)^{l}\sum_{j=k}^{l}\binom{l}{j}F^{j}(z)\sum_{n=0}^{l-j}\binom{l-j}{n}(-1)^{n+j}\left(1-F(z)\right)^{n}\\
&=\sum_{l=k}^{\infty}\frac{(-1)^{l}\Gamma(l\beta+bt)}{\Gamma(bt)\Gamma(\beta l+1)}\left(\frac{\lambda}{a^{\beta}}\right)^{l}F^{l}(z)\sum_{j=k}^{l}\binom{l}{j}(-1)^{j}\\
&=\sum_{j=k}^{\infty}\sum_{l=j}^{\infty}\binom{l}{j}(-1)^{l+j}\left(\frac{F(z)\lambda}{a^{\beta}}\right)^{l}\frac{\Gamma(l\beta+bt)}{\Gamma(bt)\Gamma(\beta l+1)}\\
&=\sum_{j=k}^{\infty}\mathrm{Pr}\{Q^{\beta}(t,\lambda F(z))=j\}=\mathrm{Pr}\{Q^{\beta}(t,\lambda F(z))\ge k\}.
	\end{align*}
	The proof follows on inserting the above expression in \eqref{orderstat}.
\end{proof}

Similarly, it can be shown that the space and the space-time fractional negative binomial process exhibit the $k$th order statistic property as in Theorem \ref{kthorder}. Also, this result holds true for a more general process, namely, the time-changed fractional
Poisson process studied by Kataria and Khandakar (2022c). 

\section{A time-changed version of $M_{f}(t)$}\label{section6}
In this section, we consider the following time-changed process:
\begin{equation}\label{more}
\mathscr{M}_{f}^{\alpha}(t)\coloneqq M_{f}(Y_{\alpha}(t))=M\left(D_{f}(Y_{\alpha}(t))\right), \ t\ge0,
\end{equation}
where the L\'evy subordinator $\{D_{f}(t)\}_{t\ge0}$, the
inverse stable subordinator $\{Y_{\alpha}(t)\}_{t\ge0}$, $0<\alpha<1$ and the GCP $\{M(t)\}_{t\ge0}$ are independent of each other. 

For $k=1$, the process $\{\mathscr{M}_{f}^{\alpha}(t)\}_{t\ge0}$ reduces to a time-changed Poisson process, that is, $\{N\left(D_{f}(Y_{\alpha}(t))\right)\}_{t\ge0}$ (see Beghin and D'Ovidio (2014), Section 3; Orsingher and Toaldo (2015), Remark 2.3). Recently, Beghin and Macci (2016) studied a multivariate version of $\{N\left(D_{f}(Y_{\alpha}(t))\right)\}_{t\ge0}$.

Let $h_{\alpha}(t,x)$ be the density of $\{Y_{\alpha}(t)\}_{t\ge0}$ and $p_{f}^{\alpha}(n,t)=\mathrm{Pr}\{\mathscr{M}_{f}^{\alpha}(t)=n\}$, $n\ge0$, be the state probabilities of $\{\mathscr{M}_{f}^{\alpha}(t)\}_{t\ge0}$. From \eqref{more}, we have
\begin{equation}\label{pfnt}
	p_{f}^{\alpha}(n,t)=\int_{0}^{\infty}p_{f}(n,x)h_{\alpha}(t,x)\mathrm{d}x.
\end{equation}
On taking Caputo fractional derivative in \eqref{pfnt} and using \eqref{pnt}, it can be shown that the state probabilities $p_{f}^{\alpha}(n,t)$ solve the following system of fractional differential equations:
	\begin{equation*}
	\frac{\mathrm{d}^{\alpha}}{\mathrm{d}t^{\alpha}}p_{f}^{\alpha}(n,t)=-f(\Lambda)p_{f}^{\alpha}(n,t)-\sum_{m=1}^{n}\sum_{ \Omega(k,m)}f^{(z_{k})}(\Lambda)\prod_{j=1}^{k}\frac{(-\lambda_{j})^{x_{j}}}{x_{j}!}p_{f}^{\alpha}(n-m,t),
	\end{equation*}
	with initial conditions 
	\begin{equation*}
	p_{f}^{\alpha}(n,0)=\begin{cases*}
	1,\ n=0,\\
	0, \ n\ge 1.
	\end{cases*}
	\end{equation*}
In view of \eqref{space}, the above system of differential equation can be written as 
\begin{equation*}
	\frac{\mathrm{d}^{\alpha}}{\mathrm{d}t^{\alpha}}p_{f}^{\alpha}(n,t)=-f\bigg(\Lambda\Big(I-\frac{1}{\Lambda}\sum_{j=1}^{n\wedge k}\lambda_{j}B^{j}\Big)\bigg)p_{f}^{\alpha}(n,t).
\end{equation*}
On using Proposition 6 of Kataria and Khandakar (2022a), the pgf of $\{\mathscr{M}_{f}^{\alpha}(t)\}_{t\ge0}$ is given by
\begin{align}\label{pgf}
\mathscr{G}_{f}^{\alpha}(u,t)&=\int_{0}^{\infty}G_{f}(u,x)h_{\alpha}(t,x)\mathrm{d}x\nonumber\\
&=\int_{0}^{\infty}\exp\bigg(-xf\Big(\sum_{j=1}^{k}\lambda_{j}(1-u^{j})\Big)\bigg)h_{\alpha}(t,x)\mathrm{d}x\nonumber\\
&=E_{\alpha,1}\bigg(-f\Big(\sum_{j=1}^{k}\lambda_{j}(1-u^{j})\Big)t^{\alpha}\bigg).
\end{align} 
It is known that the Mittag-Leffler function is an eigenfunction of the Caputo fractional derivative. So, it follows that
\begin{equation*}
\frac{\mathrm{d}^{\alpha}}{\mathrm{d}t^{\alpha}}\mathscr{G}_{f}^{\alpha}(u,t)=-f\bigg(\sum_{j=1}^{k}\lambda_{j}(1-u^{j})\bigg)\mathscr{G}_{f}^{\alpha}(u,t),\  \mathscr{G}_{f}^{\alpha}(u,0)=1.
\end{equation*}
\begin{remark}
	On taking $k=1$ and $\lambda_{1}=\lambda$, the pgf (\ref{pgf}) reduces to $\mathscr{G}_{f}^{\alpha}(u,t)\big|_{k=1}=E_{\alpha,1}\big(-f\big(\lambda(1-u)\big)t^{\alpha}\big)$ which agrees with the pgf of $\{N\left(D_{f}(Y_{\alpha}(t))\right)\}_{t\ge0}$ (see Orsingher and Toaldo (2015), Eq. (2.9)).
\end{remark}

Now, we give the mean, variance and covariance function of $\{\mathscr{M}_{f}^{\alpha}(t)\}_{t\ge0}$. We assume that the second order moments of $\{D_{f}(t)\}_{t\ge0}$ are finite. 
The mean and variance of TCGCP-I are given by $\mathbb{E}\left(M_{f}(t)\right)=\sum_{j=1}^{k}j\lambda_{j}\mathbb{E}(D_{f}(t))$ and $\operatorname{Var}(M_{f}(t))=\big(\sum_{j=1}^{k}j\lambda_{j}\big)^{2}\operatorname{Var}(D_{f}(t))+\sum_{j=1}^{k}j^{2}\lambda_{j}\mathbb{E}(D_{f}(t))$ (see  Kataria and Khandakar (2022a), Section 4.1).

 Let $q_{1}=\mathbb{E}\left(M_{f}(1)\right)$ and $q_{2}=\operatorname{Var}(M_{f}(1))$. Then, by using Theorem 2.1 of Leonenko {\it et al.} (2014), we get
\begin{align}
\mathbb{E}\left(\mathscr{M}_{f}^{\alpha}(t)\right)&=q_{1}\mathbb{E}(Y_{\alpha}(t)),\nonumber\\
\operatorname{Var}\left(\mathscr{M}_{f}^{\alpha}(t)\right)&=q_{2}\mathbb{E}(Y_{\alpha}(t))+q_{1}^{2}\operatorname{Var}(Y_{\alpha}(t)),\label{varm}\\
\operatorname{Cov}\left(\mathscr{M}_{f}^{\alpha}(s),\mathscr{M}_{f}^{\alpha}(t)\right)&=q_{2}\mathbb{E}(Y_{\alpha}(s))+q_{1}^{2}\operatorname{Cov}(Y_{\alpha}(s),Y_{\alpha}(t)),\ 0<s\le t\label{covm}.
\end{align}
As $\operatorname{Var}\left(\mathscr{M}_{f}^{\alpha}(t)\right)-\mathbb{E}\left(\mathscr{M}_{f}^{\alpha}(t)\right)>0$, $t>0$, the process $\{\mathscr{M}_{f}^{\alpha}(t)\}_{t\ge0}$ exhibits the overdispersion property.
\begin{theorem}
	The process $\{\mathscr{M}_{f}^{\alpha}(t)\}_{t\ge0}$ has the LRD property.
\end{theorem}
\begin{proof}
From \eqref{varm} and \eqref{covm}, we get	
	\begin{equation*}
	\operatorname{Corr}\left(\mathscr{M}_{f}^{\alpha}(s),\mathscr{M}_{f}^{\alpha}(t)\right)=\frac{q_{2}\mathbb{E}(Y_{\alpha}(s))+q_{1}^{2}\operatorname{Cov}(Y_{\alpha}(s),Y_{\alpha}(t))}{\sqrt{\operatorname{Var}\left(\mathscr{M}_{f}^{\alpha}(s)\right)}\sqrt{q_{2}\mathbb{E}(Y_{\alpha}(t))+q_{1}^{2}\operatorname{Var}(Y_{\alpha}(t))}}.
	\end{equation*}
On using (\ref{meani})-(\ref{covin}) for fixed $s$ and large $t$, we get
\begin{align*}
	\operatorname{Corr}\left(\mathscr{M}_{f}^{\alpha}(s),\mathscr{M}_{f}^{\alpha}(t)\right)&\sim\frac{q_{2}\mathbb{E}(Y_{\alpha}(s))+\frac{q_{1}^{2}s^{2\alpha}}{\Gamma(2\alpha+1)}}{\sqrt{\operatorname{Var}\left(\mathscr{M}_{f}^{\alpha}(s)\right)}\sqrt{\frac{q_{2}t^{\alpha}}{\Gamma(\alpha+1)}+\frac{2q_{1}^{2}t^{2\alpha}}{\Gamma(2\alpha+1)}-\frac{q_{1}^{2}t^{2\alpha}}{\Gamma^{2}(\alpha+1)}}}\\
	&\sim c(s)t^{-\alpha},
\end{align*}
	where 
	\begin{equation*}
	c(s)=\frac{q_{2}\mathbb{E}(Y_{\alpha}(s))+\frac{q_{1}^{2}s^{2\alpha}}{\Gamma(2\alpha+1)}}{\sqrt{\operatorname{Var}\left(\mathscr{M}_{f}^{\alpha}(s)\right)}\sqrt{\frac{2q_{1}^{2}}{\Gamma(2\alpha+1)}-\frac{q_{1}^{2}}{\Gamma^{2}(\alpha+1)}}}.
	\end{equation*}
	This proves the result.
\end{proof}
\subsection{GCP time-changed by fractional gamma process}
 
Beghin (2015) introduced a fractional gamma process $\{Z_{\alpha}(t)\}_{t\ge0}$ by time-changing the gamma subordinator $\{Z(t)\}_{t\ge0}$ with an independent inverse stable subordinator, that is, 
\begin{equation*}
	Z_{\alpha}(t)\coloneqq Z(Y_{\alpha}(t)).
\end{equation*}
For $x\ge0$, $t\ge0$, the following holds for its density function $g_{\alpha}(x,t)$:
\begin{equation}
	\left\{
	\begin{array}{l}
		\frac{\partial }{\partial x}g_{\alpha}(x,t)=-b(1-\mathscr{O}_{-1,t}^{\alpha})g_{\alpha}(x,t),
		\\
		g_{\alpha}(x,0)=0,\\
		\lim_{x\rightarrow +\infty }g_{\alpha}(x,t)=0,
	\end{array}
	\right.   \label{res1}
\end{equation}
where $\mathscr{O}_{-1,t}^{\alpha}$ is the fractional shift operator (see Beghin (2015), Theorem 5 and Eq. (10)). 

 For the Bern\v stein function $f_1$ given in \eqref{gamma}, the process $\{\mathscr{M}_{f}^{\alpha}(t)\}_{t\ge0}$ reduces to the following:
\begin{equation}\label{mg1}
	\mathscr{M}_{f_{1}}^{\alpha}(t)\coloneqq M(Z(Y_{\alpha}(t)))=M(Z_{\alpha}(t)), \ t\ge0.
\end{equation}

\begin{proposition}
	The pmf $p_{f_{1}}^{\alpha}(n,t)=\mathrm{Pr}\{\mathscr{M}_{f_{1}}^{\alpha}(t)=n\}$, $n\ge0$ of $\{\mathscr{M}_{f_{1}}^{\alpha}(t)\}_{t\ge0}$ satisfies the following fractional differential equations:
	\begin{equation}\label{rin}
		\mathscr{O}_{-1,t}^{\alpha}p_{f_{1}}^{\alpha}(n,t)=	(1+\Lambda/b)p_{f_{1}}^{\alpha}(n,t)-\frac{1}{b}\sum_{j=1}^{n\wedge k}\lambda_{j}p_{f_{1}}^{\alpha}(n-j,t),
	\end{equation}
	with initial conditions 
	\begin{equation*}
		p_{f_{1}}^{\alpha}(n,0)=\begin{cases*}
			1, \ n=0,\\
			0, \ n\ge1.
		\end{cases*}
	\end{equation*}
\end{proposition}
\begin{proof}
	From \eqref{mg1}, we can write
	\begin{equation}\label{jhi1}
		p_{f_{1}}^{\alpha}(n,t)=\int_{0}^{\infty}p(n,x)g_{\alpha}(x,t)\mathrm{d}x.
	\end{equation}
On both sides of \eqref{jhi1}, we apply the fractional shift operator $\mathscr{O}_{-1,t}^{\alpha}$ and use \eqref{res1} to obtain
	\begin{align*}
		\mathscr{O}_{-1,t}^{\alpha}p_{f_{1}}^{\alpha}(n,t)&=\int_{0}^{\infty}p(n,x)\Big(g_{\alpha}(x,t)+\frac{1}{b}\frac{\partial }{\partial x}g_{\alpha}(x,t)\Big)\mathrm{d}x\\
		&=p_{f_{1}}^{\alpha}(n,t)+\frac{1}{b}\left[p(n,x)g_{\alpha}(x,t)\right]_{x=0}^{\infty}-\frac{1}{b}\int_{0}^{\infty}g_{\alpha}(x,t)\frac{\mathrm{d}}{\mathrm{d}x}p(n,x)\mathrm{d}x\\
		&=p_{f_{1}}^{\alpha}(n,t)-\frac{1}{b}\int_{0}^{\infty}g_{\alpha}(x,t)\Big(-\Lambda p(n,x)+	\sum_{j=1}^{n\wedge k}\lambda_{j}p(n-j,x)\Big)\mathrm{d}x,
	\end{align*}
 where in the last step we have used \eqref{cre} with $\beta=1$. This coincides with the desired result by considering \eqref{jhi1}.
\end{proof}
\subsection{Space-time fractional version of the GCP}

 We take the Bern\v stein function $f_2$ in \eqref{more}  to obtain the space-time fractional version of the GCP, namely, the generalized space-time fractional counting process (GSTFCP) $\{\mathscr{M}_{f_{2}}^{\alpha}(t)\}_{t\ge0}$, that is, 
 \begin{equation*}
 	\mathscr{M}_{f_{2}}^{\alpha}(t)\coloneqq M(D_{\beta}(Y_{\alpha}(t))), \ t\ge0.
 \end{equation*}
 For $1\le j \le k$, if $\lambda_{j}=\lambda$ and $\lambda_{j}=\lambda(1-\rho)\rho^{j-1}/(1-\rho^{k})$, $0\le \rho<1$, then the GSTFCP reduces to the space-time fractional version of PPoK and PAPoK, repectively.
 
  From \eqref{pgf}, the pgf of GSTFCP is given by
	\begin{equation}\label{pgfgstfcp}
		\mathscr{G}_{f_{2}}^{\alpha}(u,t)=E_{\alpha,1}\bigg(-\Big(\sum_{j=1}^{k}(1-u^{j})\lambda_{j}\Big)^{\beta}t^{\alpha}\bigg).
	\end{equation}
For $k=1$, it reduces to the pgf of STFPP (see Orsingher and Polito (2012), Eq. (2.28)). It solves the following differential equations:
	 \begin{equation*}
		\frac{\mathrm{d}^{\alpha}}{\mathrm{d}t^{\alpha}}\mathscr{G}_{f_{2}}^{\alpha}(u,t)=-\bigg(\sum_{j=1}^{k}(1-u^{j})\lambda_{j}\bigg)^{\beta}\mathscr{G}_{f_{2}}^{\alpha}(u,t), \ \mathscr{G}_{f_{2}}^{\alpha}(u,0)=1.
	\end{equation*}
Let $X_{i}$, $i\ge1$ be iid uniform random variables in $[0,1]$ and $\{N_{\alpha}(t,\Lambda^{\beta})\}_{t\ge0}$ be the TFPP with intensity $\Lambda^{\beta}$ such that  $\min_{0\le i \le N_{\alpha}(t,\Lambda^{\beta})} X_{i}^{1/\beta}= 1$ when $N_{\alpha}(t,\Lambda^{\beta})=0$. Then,
for $0<u<1$, the pgf of GSTFCP can be written as 
\begin{align}\label{uniform}
\mathscr{G}_{f_{2}}^{\alpha}(u,t)&=\sum_{m=0}^{\infty}\frac{(-t^{\alpha}\Lambda^{\beta})^{m}}{\Gamma(m\alpha+1)}\Big(1-\frac{1}{\Lambda}\sum_{j=1}^{k}\lambda_{j}u^{j}\Big)^{\beta m}\nonumber\\
&=\sum_{m=0}^{\infty}\frac{(-t^{\alpha}\Lambda^{\beta})^{m}}{\Gamma(m\alpha+1)}\sum_{r=0}^{m}(-1)^{r}\binom{m}{r}\bigg(1-\Big(1-\frac{1}{\Lambda}\sum_{j=1}^{k}\lambda_{j}u^{j}\Big)^{\beta}\bigg)^{r}\nonumber\\
&=\sum_{r=0}^{\infty}\bigg(1-\Big(1-\frac{1}{\Lambda}\sum_{j=1}^{k}\lambda_{j}u^{j}\Big)^{\beta}\bigg)^{r}\sum_{m=r}^{\infty}(-1)^{m-r}\binom{m}{r}\frac{(t^{\alpha}\Lambda^{\beta})^{m}}{\Gamma(m\alpha+1)}\nonumber\\
	&=\sum_{r=0}^{\infty}\bigg(\mathrm{Pr}\Big\{X_{r}^{1/\beta}\ge 1-\frac{1}{\Lambda}\sum_{j=1}^{k}\lambda_{j}u^{j}\Big\}\bigg)^{r}\mathrm{Pr}\{N_{\alpha}(t,\Lambda^{\beta})=r\}\nonumber\\
	&=\mathrm{Pr}\bigg\{\min_{0\le r\le N_{\alpha}(t,\Lambda^{\beta})} X_{r}^{1/\beta}\ge 1-\frac{1}{\Lambda}\sum_{j=1}^{k}\lambda_{j}u^{j}\bigg\}.
\end{align}

The result obtained in \eqref{uniform} is of particular interest as it connects
the GSTFCP with TFPP via uniform random variables. On using the generalized binomial theorem to expand the right hand side of \eqref{pgfgstfcp}, we obtain the pmf of GSTFCP in the following form: 
\begin{equation}\label{pmff2}
	\mathrm{Pr}\{\mathscr{M}_{f_{2}}^{\alpha}(t)=n\}=\sum_{ \Omega(k,n)}\prod_{j=1}^{k}\frac{\lambda_{j}^{x_{j}}}{x_{j}!}\frac{(-1)^{z_{k}}}{\Lambda^{z_{k}}}\sum_{m=0}^{\infty}\frac{(-\Lambda^{\beta}t^{\alpha})^{m}}{\Gamma(\alpha m+1)}\frac{\Gamma(\beta m+1)}{\Gamma(\beta m+1-z_{k})},\ n\ge0.
\end{equation}
The steps involved to obtain \eqref{pmff2} are similar to the one involved in the proof of \eqref{space}.

		Next, we obtain an alternate form of the pmf of GSTFCP. The pmf \eqref{p(n,t)} of GCP can be re-written as follows:
		\begin{equation*}
			p(n,t)=\sum_{\Omega(k,n)}\prod_{j=1}^{k}\frac{\lambda_{j}^{x_{j}}}{x_{j}!}(-\partial_{\Lambda})^{z_{k}}e^{-\Lambda t}, \ n\ge 0.
		\end{equation*}
		Then, from \eqref{more}, we have
		\begin{align}
			\mathrm{Pr}\{\mathscr{M}_{f_{2}}^{\alpha}(t)=n\}&=\int_{0}^{\infty}p(n,s)\mathrm{Pr}\{D_{\beta}(Y_{\alpha}(t))\in\mathrm{d}s\}\nonumber\\
			&=\int_{0}^{\infty}\sum_{\Omega(k,n)}\prod_{j=1}^{k}\frac{\lambda_{j}^{x_{j}}}{x_{j}!}(-\partial_{\Lambda})^{z_{k}}e^{-\Lambda s}\mathrm{Pr}\{D_{\beta}(Y_{\alpha}(t))\in\mathrm{d}s\}\nonumber\\
			&=\sum_{\Omega(k,n)}\prod_{j=1}^{k}\frac{\lambda_{j}^{x_{j}}}{x_{j}!}(-\partial_{\Lambda})^{z_{k}}\mathbb{E}\left(e^{-\Lambda D_{\beta}(Y_{\alpha}(t))}\right)\label{pmfstfpp1}\\
			&=\sum_{\Omega(k,n)}\prod_{j=1}^{k}\frac{\lambda_{j}^{x_{j}}}{x_{j}!}(-\partial_{\Lambda})^{z_{k}}E_{\alpha,1}\left(-\Lambda^{\beta}t^{\alpha}\right)\label{alternatepmf},
		\end{align}
		where in the last step we have used Lemma 3.1 of  Beghin and D'Ovidio (2014). Thus, the pgf of GSTFCP can alternatively be obtained as follows:
	\begin{align*}
		\mathscr{G}_{f_{2}}^{\alpha}(u,t)&=\sum_{n=0}^{\infty}u^{n}\mathrm{Pr}\{\mathscr{M}_{f_{2}}^{\alpha}(t)=n\}\\
		&=\sum_{n=0}^{\infty}u^{n}\sum_{\Omega(k,n)}\prod_{j=1}^{k}\frac{\lambda_{j}^{x_{j}}}{x_{j}!}(-\partial_{\Lambda})^{z_{k}}\mathbb{E}\left(e^{-\Lambda D_{\beta}(Y_{\alpha}(t))}\right),\ (\text{using}\ \eqref{pmfstfpp1})\\
		&=\sum_{r=0}^{\infty}\Big(\sum_{j=1}^{k}\lambda_{j}u^{j}\Big)^{r}\frac{(-\partial_{\Lambda})^{r}}{r!}\mathbb{E}\left(e^{-\Lambda D_{\beta}(Y_{\alpha}(t))}\right)\\
		&=\mathbb{E}\big(e^{-\sum_{j=1}^{k}(1-u^{j})\lambda_{j}D_{\beta}(Y_{\alpha}(t))}\big),
	\end{align*}
	which reduces to \eqref{pgfgstfcp} on using Lemma 3.1 of Beghin and D'Ovidio (2014).

Let $\{N^{\alpha,\beta}(t)\}_{t\ge0}$ denotes the STFPP with intensity $\lambda$. For notational convenience, let $N^{\alpha,\beta}(t, \lambda)=N^{\alpha,\beta}(t)$. On taking $k=1$  and $\Lambda=\lambda_{1}=\lambda$ (say) in \eqref{pmff2}, we get 
\begin{equation}\label{pmfstfpp}
	\mathrm{Pr}\{N^{\alpha,\beta}(t,\lambda)=n\}=\frac{(-1)^{n}}{n!}\sum_{m=0}^{\infty}\frac{(-\lambda^{\beta}t^{\alpha})^{m}\Gamma(\beta m+1)}{\Gamma(\beta m+1-n)\Gamma(\alpha m+1)},
\end{equation}
which agrees with the pmf of STFPP (see Orsingher and Polito (2012), Eq. (2.29)).
 Again taking $k=1$ and $\Lambda=\lambda$ in \eqref{alternatepmf}, we get an alternate version of the pmf of STFPP (see Beghin and D'Ovidio (2014), Eq. (3.19)).
\begin{theorem}\label{statsticst}
	Let $X_{(k)}^{N^{\alpha,\beta}(t,\lambda)}$ be a $k$th order statistic from the sample of size $N^{\alpha,\beta}(t,\lambda)$, $k\in \{1,2,\dots, N^{\alpha,\beta}(t,\lambda)\}$ and $\{X_{i}\}_{i\ge1}$ be a sequence of iid random variables with distribution function $F$. Then, 
	\begin{equation*}
		\mathrm{Pr}\{X_{(k)}^{N^{\alpha,\beta}(t,\lambda)}<z\big|N^{\alpha,\beta}(t,\lambda)\ge k\}=\frac{\mathrm{Pr}\{N^{\alpha,\beta}(t,\lambda F(z))\ge k\}}{\mathrm{Pr}\{N^{\alpha,\beta}(t,\lambda)\ge k\}}.	
	\end{equation*}
\end{theorem}
\begin{proof}
	On using the conditional probability law, we get
\small	\begin{equation}\label{orderstat1}
		\mathrm{Pr}\{X_{(k)}^{N^{\alpha,\beta}(t,\lambda)}<z\big|N^{\alpha,\beta}(t,\lambda)\ge k\}
		=\frac{\sum_{n=k}^{\infty}\mathrm{Pr}\{X_{(k)}^{N^{\alpha,\beta}(t,\lambda)}<z\big|N^{\alpha,\beta}(t,\lambda)=n\}	\mathrm{Pr}\{N^{\alpha,\beta}(t,\lambda)=n\}}{\mathrm{Pr}\{N^{\alpha,\beta}(t,\lambda)\ge k\}}.
	\end{equation}
\normalsize
Using \eqref{stan} and \eqref{pmfstfpp}, the quantity in the numerator of the right hand side of \eqref{orderstat1} can be evaluated as follows:
	\begin{align*}
		\sum_{n=k}^{\infty}&\mathrm{Pr}\{X_{(k)}^{N^{\alpha,\beta}(t,\lambda)}<z\big|N^{\alpha,\beta}(t,\lambda)=n\}	\mathrm{Pr}\{N^{\alpha,\beta}(t,\lambda)=n\}\\
		&=\sum_{n=k}^{\infty}\sum_{j=k}^{n}\binom{n}{j}F^{j}(z)\left(1-F(z)\right)^{n-j}\frac{(-1)^{n}}{n!}\sum_{m=0}^{\infty}\frac{(-\lambda^{\beta}t^{\alpha})^{m}\Gamma(\beta m+1)}{\Gamma(\beta m+1-n)\Gamma(\alpha m+1)}\\
		&=\sum_{j=k}^{\infty}\sum_{m=0}^{\infty}\sum_{n=j}^{\infty}\frac{F^{j}(z)}{(n-j)!j!}\left(1-F(z)\right)^{n-j}\frac{(-1)^{n}(-\lambda^{\beta}t^{\alpha})^{m}\Gamma(\beta m+1)}{\Gamma(\beta m+1-n)\Gamma(\alpha m+1)}\\
		&=\sum_{j=k}^{\infty}\frac{(-1)^{j}F^{j}(z)}{j!}\sum_{m=0}^{\infty}\frac{(-\lambda^{\beta}t^{\alpha})^{m}\Gamma(\beta m+1)}{\Gamma(\alpha m+1)}\sum_{n=0}^{\infty}\frac{(-1)^{n}\left(1-F(z)\right)^{n}}{n!\Gamma(\beta m+1-n-j)}\\
		&=\sum_{j=k}^{\infty}\frac{(-1)^{j}}{j!}\sum_{m=0}^{\infty}\frac{(-\lambda^{\beta}t^{\alpha})^{m}\Gamma(\beta m+1)}{\Gamma(\alpha m+1)}\frac{\left(F(z)\right)^{\beta m}}{\Gamma(\beta m+1-j)}\\
		&=\sum_{j=k}^{\infty}\frac{(-1)^{j}}{j!}\sum_{m=0}^{\infty}\frac{(-\lambda^{\beta}F^{\beta}(z)t^{\alpha})^{m}\Gamma(\beta m+1)}{\Gamma(\alpha m+1)\Gamma(\beta m+1-j)}\\
		&=\sum_{j=k}^{\infty}\mathrm{Pr}\{N^{\alpha,\beta}(t,\lambda F(z))=j\}=\mathrm{Pr}\{N^{\alpha,\beta}(t,\lambda F(z))\ge k\}.
	\end{align*}
		The proof follows on inserting the above expression in \eqref{orderstat1}.
\end{proof}
\begin{remark}
	Polito and Scalas (2016) studied a time-changed Poisson process that generalizes the STFPP. Its state probabilities (see Polito and Scalas (2016), Eq. (3.19)) satisfy a system of differential equations that involves Prabhakar derivative. It can be shown that this process also has the $k$th order statistic property, as discussed in Theorem \ref{statsticst}.
\end{remark}

\end{document}